\documentclass[12pt]{article}
\usepackage{amsmath,amssymb,amsfonts,amsthm,mdframed,graphicx}
\usepackage{color, colortbl}
\usepackage{xcolor}
\usepackage{array}
\usepackage{multirow}

\usepackage{tikz}
\usepackage[colorlinks,
linkcolor=blue,
anchorcolor=blue,
citecolor=blue
]{hyperref}

\newtheorem{thm}{Theorem}[section]

\newtheorem{prop}[thm]{Proposition}

\newtheorem{coro}[thm]{Corollary}

\definecolor{blue}{rgb}{0,0,1}
{}

\title{On permutations avoiding partially ordered patterns defined by bipartite graphs}

\author{Sergey Kitaev\footnote{Department of Mathematics and Statistics, University of Strathclyde, 26 Richmond Street, Glasgow G1 1XH, United Kingdom. 
{\bf Email:} sergey.kitaev@strath.ac.uk.}\ \ and Artem Pyatkin\footnote{Sobolev Institute of Mathematics, Koptyug ave, 4, Novosibirsk, 630090, Russia.  {\bf Email:} artem@math.nsc.ru.}}

\begin{document}

\maketitle

\begin{abstract} Partially ordered patterns (POPs) generalize the notion of classical patterns studied in the literature in the context of permutations, words, compositions and partitions. In this paper, we give a number of general, and specific enumerative results for POPs in permutations defined by bipartite graphs, substantially extending the list of known results in this direction. In particular, we completely characterize the Wilf-equivalence for patterns defined by the N-shape posets. \\

\noindent
{\bf Keywords:} permutation pattern, partially ordered pattern, bipartite graph, enumeration, Wilf-equivalence, N-pattern\\

\noindent 
{\bf AMS Subject Classifications:}  05A05, 05A15
\end{abstract}

\section{Introduction}\label{intro-sec}

An occurrence of a (classical) permutation pattern $p=p_1\cdots p_k$ in a permutation $\pi=\pi_1\cdots\pi_n$ is a subsequence $\pi_{i_1}\cdots\pi_{i_k}$, where $1\leq i_1<\cdots< i_k\leq n$, such that $\pi_{i_j}<\pi_{i_m}$ if and only if $p_j<p_m$. For example, the permutation $364125$ has two occurrences of the pattern 123, namely, the subsequences 345 and 125, while this permutation {\em avoids} (that is, has no occurrences of) the pattern 4321. Permutation patterns are an active area of research that attracts much attention in the literature (e.g.\ see \cite{ALLV2016,BBEP2020,CGZ18,M2020} and references therein). 

A {\em partially ordered pattern} ({\em POP}) $p$ of length $k$ is defined by a $k$-element partially ordered set (poset) $P$ labeled by the elements in $[k]:=\{1,\ldots,k\}$. An occurrence of such a POP $p$ in a permutation $\pi=\pi_1\cdots\pi_n$ is a subsequence $\pi_{i_1}\cdots\pi_{i_k}$, where $1\leq i_1<\cdots< i_k\leq n$,  such that $\pi_{i_j}<\pi_{i_m}$ if  $j<m$ in $P$. Thus, a classical pattern of length $k$ corresponds to a $k$-element chain. For example, the POP $p=$ \hspace{-3.5mm}
\begin{minipage}[c]{3.5em}\scalebox{1}{
\begin{tikzpicture}[scale=0.5]

\draw [line width=1](0,-0.5)--(0,0.5);

\draw (0,-0.5) node [scale=0.4, circle, draw,fill=black]{};
\draw (1,-0.5) node [scale=0.4, circle, draw,fill=black]{};
\draw (0,0.5) node [scale=0.4, circle, draw,fill=black]{};

\node [left] at (0,-0.6){${\small 3}$};
\node [right] at (1,-0.6){${\small 2}$};
\node [left] at (0,0.6){${\small 1}$};

\end{tikzpicture}
}\end{minipage}
occurs four times in the permutation 41253, namely, as the subsequences 412, 413, 423, and 453, as in each of these subsequences the element in position 1 is bigger than that in position 3. Clearly, avoiding the pattern $p$ is the same as avoiding the patterns 312, 321	and 231 at the same time, so POPs provide a convenient language to deal with larger sets of permutation patterns. Moreover, the notion of a POP provides a uniform notation for several combinatorial structures such as peaks, valleys, modified maxima and minima, $p$-descents in permutations, and many others~\cite{Kit4}.

A POP is {\em bipartite} if the underlying graph of the poset defining it is a bipartite graph. Clearly, a POP is bipartite if and only if in the Hasse diagram of the poset defining it the longest increasing chain has length 1.

Permutations of length $n$ are called $n$-permutations, and $S_n$ denotes the set of all $n$-permutations. For an $n$-permutation $\pi$, the {\em complement} $c(\pi)$ of $\pi$ is obtained from $\pi$ by replacing each element $x$ by $n+1-x$. For example, $c(3142)=2413$. The same operation is well-defined on labels in $[n]$ of an $n$-element poset. Also, the {\em reverse} $r(\pi)$ of $\pi$ is obtained by writing the elements of $\pi$ in the reverse order. For example, $r(52413)=31425$. The complement, reverse, and usual group theoretical inverse are known as {\em trivial bijections} from $S_n$ to $S_n$. Finally, for a sequence $s$ of numbers, the {\em reduced form} red($s$) of $s$ is obtained from $s$ by replacing in it the $i$-th smallest element by $i$ for $i=1,2,\ldots$. For example, red$((2,6,9,1,4))=(2,4,5,1,3)$ and in the context of this paper, commas and brackets will be omitted. 

Denote by $S_n(p)$ the set of $n$-permutations avoiding $p$. Patterns $p_1$ and $p_2$ are {\em Wilf-equivalent} if $|S_n(p_1)|=|S_n(p_2)|$ for every $n\geq 0$. Throughout this paper, the notation $a(n)$ is used for the number of $n$-permutations avoiding 
the pattern $p$ in question, that is, $a(n)=|S_n(p)|$. 

POPs were introduced in \cite{Kit1}, and they were studied in the context of permutations, words, compositions and partitions (see references in \cite{GK19}).   A systematic search for connections between sequences in the {\em Online Encyclopedia of Integer Sequences} ({\em OEIS}) \cite{oeis} and permutations avoiding POPs of length 4 and 5 was conducted in \cite{GK19}, which resulted in 13 new enumerative results for classical patterns of length 4 and 5. Five out of 15 open problems stated in \cite{GK19} were answered in \cite{YWZ21}. The most relevant to our paper are the following general results about bipartite POPs obtained in \cite{GK19}, where ``g.f.'' stands for ``generating function''. 

\begin{figure}[htbp]
  \centering
\begin{tikzpicture}[scale=0.8]

\draw [line width=1](0,0)--(2,1.5);
\draw [line width=1](1,0)--(2,1.5);
\draw [line width=1](2,0)--(2,1.5);
\draw [line width=1](4,0)--(2,1.5);

\draw (0,0) node [scale=0.4, circle, draw,fill=black]{};
\draw (1,0) node [scale=0.4, circle, draw,fill=black]{};
\draw (2,0) node [scale=0.4, circle, draw,fill=black]{};
\draw (4,0) node [scale=0.4, circle, draw,fill=black]{};
\draw (2,1.5) node [scale=0.4, circle, draw,fill=black]{};

\node [below] at (0,0){$x_2$};
\node [below] at (1,0){$x_3$};
\node [below] at (2,0){$x_4$};
\node [below] at (4,0){$x_k$};
\node [above] at (2,1.5){$x_1$};

\node [left] at (0,1.5){$a)$};
\node [left] at (7,1.5){$b)$};

\draw (2.5,0.2) node [scale=0.15, circle, draw,fill=black]{};
\draw (2.75,0.2) node [scale=0.15, circle, draw,fill=black]{};
\draw (3,0.2) node [scale=0.15, circle, draw,fill=black]{};

\draw [line width=1](7,0)--(8,1.5);
\draw [line width=1](8,0)--(8,1.5);
\draw [line width=1](12,0)--(8,1.5);
\draw [line width=1](7,0)--(11,1.5);
\draw [line width=1](8,0)--(11,1.5);
\draw [line width=1](12,0)--(11,1.5);

\draw (7,0) node [scale=0.4, circle, draw,fill=black]{};
\draw (8,0) node [scale=0.4, circle, draw,fill=black]{};
\draw (12,0) node [scale=0.4, circle, draw,fill=black]{};
\draw (8,1.5) node [scale=0.4, circle, draw,fill=black]{};
\draw (11,1.5) node [scale=0.4, circle, draw,fill=black]{};

\node [below] at (7,-0.1){$2$};
\node [below] at (8,-0.1){$3$};
\node [below] at (12,-0.1){$ k-1$};
\node [above] at (8,1.5){$1$};
\node [above] at (11,1.5){$k$};

\draw (9,0.1) node [scale=0.15, circle, draw,fill=black]{};
\draw (9.25,0.1) node [scale=0.15, circle, draw,fill=black]{};
\draw (9.5,0.1) node [scale=0.15, circle, draw,fill=black]{};

\end{tikzpicture}
\caption{The POPs in Theorems~\ref{thm-B1} and~\ref{thm-B2}.}
 \label{pic-B1}
\end{figure}
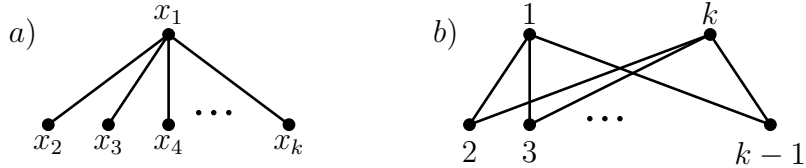

\begin{thm}[\cite{GK19}]\label{thm-B1} Let $p$ be the POP in Figure~\ref{pic-B1}a, where 
 $\{x_1,\ldots,x_k\}=[k]$ and $k\geq 1$. Then, 
$$a(n)=\left\{ \begin{array}{ll} 
n! & \mbox{if }n<k\\ 
(k-1)!(k-1)^{n-k+1} & \mbox{if } n\geq k.\end{array}\right.$$
Also, the respective g.f.\ is
$$\sum_{n\geq 0}a(n)x^n=\frac{(k-1)(k-1)!x^k}{1-(k-1)x}+\sum_{i=0}^{k-1}i!x^i.$$
\end{thm}







\begin{thm}[\cite{GK19}]\label{thm-B2} Let $p$ be the POP in Figure~\ref{pic-B1}b, where 
 $k\geq 2$ (if $k=2$ then $p$ is two independent elements). Then, $$a(n)=\left\{ \begin{array}{ll} 
n! & \mbox{if }n<k\\ 
2(k-2) \cdot a(n-1) - (k-2)(k-3)\cdot a(n-2) & \mbox{if } n\geq k.\end{array}\right.$$
Also, the respective g.f.\ is
$$\sum_{n\geq 0}a(n)x^n=\frac{A(x)-B(x)+C(x)}{1-2(k-2)x + (k-2)(k-3)x^2},$$
where $$A(x) = \sum_{i=0}^{k-3}i!x^i,\  B(x)= 2(k-2)\sum_{i=0}^{k-4}i!x^{i+1}, C(x)=(k-2)(k-3)\sum_{i=0}^{k-5}i!x^{i+2}.$$
 \end{thm}

The following simple result is also useful. 

\begin{thm}[\cite{GK19}]\label{trivial-sym-thm} Let $p$ be a POP of length $k$. Also, let $p'$ be the POP obtained from $p$ by applying the complement operation on its labels, that is, by replacing a label $x$ by $k+1-x$. Moreover, let  $p''$ be the POP obtained from $p$ by flipping the poset with respect to a horizontal line. Then, $|S_n(p)|=|S_n(p')|=|S_n(p'')|$ for any $n\geq 0$, that is, $p$, $p'$ and $p''$ are Wilf-equivalent. \end{thm}

In this paper, we extend the results in Theorems~\ref{thm-B1} and~\ref{thm-B2} to other bipartite POPs and discuss certain generalizations for some of them. The paper is organized as follows. In Section~\ref{complete-bipartite-POPs}, we derive enumerative results about certain classes of POPs defined by complete bipartite graphs, in particular, enumerating all avoidance classes for such POPs of length at most 5. In Section~\ref{WC-N-patterns-sec}, we give a complete classification of the Wilf-equivalence for N-patterns, which are the bipartite POPs of length 4 defined by the posets of N-shape. We prove that there are three Wilf-equivalence classes for N-patterns, and enumerate all classes. In Section~\ref{sticks-sec}, we discuss enumeration of so-called DC POPs defined by disjoint chains (including those of length 2). Finally, in Section~\ref{final-sec} we discuss directions for further research.

\section{POPs defined by complete bipartite graphs}\label{complete-bipartite-POPs}

We denote by $p=(A/B)$, where $A\subset [k]$ and $B=[k] \setminus A$, the POP $p$ of length $k$ in which each element corresponding to a position from $A$ must be larger than all elements that are in positions from $B$.  We call such patterns {\it complete bipartite POPs}. We are interested in the number of $n$-permutations avoiding a given complete bipartite POP $p$ of length $k$.  Clearly, for $n<k$ there are $n!$ $p$-avoiding $n$-permutations, so in our arguments below we can assume that $n\ge k$. Note that the cases of $|A|=1$ and $A=\{1,k\}$ are given by Theorems~\ref{thm-B1} and~\ref{thm-B2}, respectively. 

In this section, we prove that the number $a(n)$ of $n$-permutations avoiding POPs with $A\in \{ \{i,i+1\}, \{ j,j+2\}\}$, where $i\in [n-1]$ and $j\in [n-2]$, satisfies the same recurrence relation as that in Theorem~\ref{thm-B2}.

Given an $n$-permutation $\pi$ and $t\in [n]$, denote by $\pi'=\pi \setminus t$ the $(n-1)$- permutation obtained from $\pi$ by deleting the element $t$ and decreasing by 1 each element in $\pi$ larger than $t$. For any POP $p$, we say that $\pi$ and $\pi'$ are $p$-{\it equivalent} if either both of the permutations contain an occurrence of $p$, or both of them are $p$-avoiding. For example, let $p=(A/B)$ be a complete bipartite POP of length $k$, where $|A|=s$ and $1\in A$. If $\pi=\pi_1\cdots\pi_n\in S_n$ and $\pi_1=t$, where $t\in [k-s]$, then $\pi$ and $\pi'=\pi\setminus t$ are $p$-equivalent, since the first element of $\pi$ clearly cannot be a part of $p$. It is straightforward to extend the notion of $p$-equivalence to the situation when more than one element are removed from a permutation. 

Note that for establishing $p$-equivalence, it is sufficient to prove the implication: ``If $\pi$ contains $p$ then $\pi'$ contains $p$'', since the opposite implication is trivial taking into account the hereditary nature of POP containment.

Our next theorem gives the same recurrence relation, and hence the same g.f., as that in Theorem~\ref{thm-B2}. In what follows, we use the notation $[x,y]:=\{x,x+1,\ldots,y\}$ and remind that $[x]=[1,x]$.

\begin{thm}\label{t3}
Let $p=(A/B)$ be a complete bipartite POP of length $k$. If $A=\{i,i+1\}$, where $i\in[k-1]$ for $k\geq 3$, then 
$$a(n)=\left\{ \begin{array}{ll} 
n! & \mbox{if }n<k\\ 
2(k-2)\cdot a(n-1)-(k-2)(k-3)\cdot a(n-2)& \mbox{if } n\geq k.\end{array}\right.$$
\end{thm}

\begin{proof} Let $\pi=\pi_1\cdots\pi_n\in S_n$ and $\pi_s=n$ and $\pi_t=n-1$ for some $s,t\in [n]$. Clearly, if both $s,t\in [i,n-k+i+1]$ then $\pi$ contains $p$. 
On the other hand, if $s\in [i-1]\cup [n-k+i+2,n]$ (respectively, $t\in [i-1]\cup [n-k+i+2,n]$) then $n=\pi_s$ (respectively, $n-1=\pi_t$) cannot be a part of $p$, and thus, $\pi$ and $\pi\setminus n$ (respectively, $\pi\setminus (n-1)$) are $p$-equivalent. In total, there are $2(k-2)\cdot a(n-1)$ such $p$-avoiding permutations. Among them, there are $(k-2)(k-3)\cdot a(n-2)$ $p$-avoiding permutations having both $s,t\in [i-1]\cup [n-k+i+2,n]$ that were counted twice (note that in this case both $n$ and $n-1$ cannot be a part of $p$ and hence $\pi$ and $\pi\setminus \{n-1,n\}$ are $p$-equivalent). In total, we have
$$a(n)=2(k-2)\cdot a(n-1)-(k-2)(k-3)\cdot a(n-2),$$
as required.
\end{proof}

Generalizing the proof of Theorem~\ref{t3} in a straightforward way (using the principle of inclusion-exclusion) gives the following result.

\begin{thm}\label{thm-general-1} Let $p=(A/B)$ be a complete bipartite POP of length $k$. For $A=[i,i+j]$, where $i\in[k-j]$ and $k\geq j+2$,
$a(n)=n!$ if $n<k$, and for $n\geq k$,
$$a(n)=(j+1)(k-j-1)\cdot a(n-1)-\binom{j+1}{2}(k-j-1)(k-j-2)\cdot a(n-2)$$
$$+\binom{j+1}{3}(k-j-1)(k-j-2)(k-j-3)\cdot a(n-3)+\cdots$$
$$+(-1)^j(k-j-1)(k-j-2)\cdots (k-2j-1)\cdot a(n-j-1).$$
In particular, for $A=\{i,i+1,i+2\}$ (the case of $j=2$), we have 
$$a(n)=3(k-3)\cdot a(n-1)-3(k-3)(k-4)\cdot a(n-2)+(k-3)(k-4)(k-5)\cdot a(n-3).$$
\end{thm}

\begin{proof} Let $\pi=\pi_1\cdots\pi_n\in S_n$. Let us say that an element $x\in \pi$ is {\it large} if it is among the $j+1$ largest elements, i.e.\ $x\in [n-j,n]$.
Clearly, if the positions of all large elements in $\pi$ are in $[i,n-k+i+j]$ then $\pi$ contains $p$. On the other hand, 
any large element whose position is in $C:=[i-1]\cup [n-k+i+j+1,n]$ cannot be a part of $p$. 
So, applying the principle of inclusion-exclusion, we count the number of $p$-avoiding permutations with one large element in $C$, reduce that number by the number of $p$-avoiding permutations with two large elements in $C$, add the number of $p$-avoiding permutations with three large elements in $C$, etc. In order to count the number of $p$-avoiding permutations with $\ell$ large elements in $C$, note that there are $\binom{j+1}{\ell}$ ways to choose $\ell$ large elements. The first of these elements has $(k-j-1)$ possible positions in $C$, the second one has just $(k-j-2)$ positions, and so on, the $\ell$-th element has $(k-j-\ell)$ possible positions. Removing these $\ell$ large elements we obtain a $p$-equivalent permutation of length $n-\ell$. Hence, the total  number of $p$-avoiding permutations with $\ell$ large elements in $C$ is
$$\binom{j+1}{\ell}(k-j-1)(k-j-2)\cdots (k-j-\ell)\cdot a(n-\ell).$$
Summing up these summands with a multiple of $(-1)^{\ell-1}$ for all $\ell$ from $1$ to $j+1$ gives the required formula.
\end{proof}


Note that in the case of $j=0$ in Theorem~\ref{thm-general-1}, we deal with the POP in Figure~\ref{pic-B1}a, and the recurrence given by Theorem~\ref{thm-general-1} in this case, namely, $a(n)=(k-1)a(n-1)$, can be used to derive the formula in Theorem~\ref{thm-B1}.

\begin{thm}\label{t4}
Let $p=(A/B)$ be a complete bipartite POP of length $k$. If $A=\{i,i+2\}$, where $i\in[k-2]$ and $k\geq 3$, then 
$$a(n)=\left\{ \begin{array}{ll} 
n! & \mbox{if }n<k\\ 
2(k-2)\cdot a(n-1)-(k-2)(k-3)\cdot a(n-2) & \mbox{if } n\geq k.\end{array}\right.$$
\end{thm}

\begin{proof} Let $\pi=\pi_1\cdots\pi_n\in S_n$ and $n=\pi_s$ and $n-1=\pi_t$ for some $s,t\in [n]$. Further, we let $C:=[i-1]\cup [n-k+i+3,n]$. Similarly to the proof of Theorem~\ref{t3}, if $n$ or $n-1$ occurs in $\pi$ in a position from $C$ then it cannot be a part of $p$ and hence there are $2(k-3)\cdot a(n-1)-(k-3)(k-4)\cdot a(n-2)$ such $p$-avoiding $n$-permutations.  Assume $s,t\not\in C$. Then $s=t\pm 1$ since otherwise $\pi$ contains $p$. Let $\pi'=\pi\setminus n$. We next show that $\pi$ and $\pi'$ are $p$-equivalent. Indeed, assume that $\pi$ contains an occurrence of $p$. Then, this occurrence cannot involve both $n$ and $n-1$ since they are adjacent in $\pi$ and hence cannot correspond to $i$ and $i+2$ in $p$. Therefore, if $n$ is involved in an occurrence of $p$ in $\pi$ then substituting it by $n-1$ gives an occurrence of $p$ in $\pi'$. Note that $\pi'$ has a property that the position of its maximum element  must be in $[i,n-k+i+1]$. Clearly, there are 
$a(n-1)-(k-3)\cdot a(n-2)$ such $p$-avoiding $(n-1)$-permutations $\pi'$, and each of them corresponds to two $p$-avoiding $n$-permutations $\pi$ ($n$ can go straight before or straight after $n-1$). Thus, we have
$$a(n)=2(k-3)\cdot a(n-1)+(k-3)(k-4)\cdot a(n-2)+2[a(n-1)-(k-3)\cdot a(n-2)]$$
$$=2(k-2)\cdot a(n-1)-(k-2)(k-3)\cdot a(n-2),$$
as required.
\end{proof} 

Note that from Theorems~\ref{t3} and~\ref{t4}, all complete bipartite POPs of length 4 with $|A|=2$ are Wilf-equivalent, and their counting sequence begins with 1, 2, 6, 20, 68, 232, 792, 2704, 9232, ... (this is the sequence A006012 in \cite{oeis} with the g.f.\ $\sum_{n\geq 0}a(n)x^n=\frac{1-3x}{1-4x + 2x^2}$ and several combinatorial interpretations, including two related to permutations). As for length 5, all
POPs with $|A|=2$ are also Wilf-equivalent except for $A\in\{\{ 1,4\},\{2,5\}\}$, and these two POPs are Wilf-equivalent by Theorem~\ref{trivial-sym-thm}. In the ``generic case'', the counting sequence begins with 1, 2, 6, 24, 108, 504, 2376, 11232, 53136, ... (the sequence A094433 in \cite{oeis}), and in the ``exceptional cases'', the counting sequence begins with 1, 2, 6, 24, 108, 504, 2364, 11052, 51456, ... .

The following theorem completes enumeration of complete bipartite POPs of length~5, namely, it enumerates avoidance for the POP $(\{1,4\},\{2,3,5\})=$\hspace{-0.4cm}\begin{minipage}[c]{3.5em}\scalebox{1}
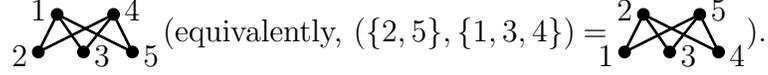
{
\begin{tikzpicture}[scale=0.5]

\draw [line width=1](0,-0.5)--(0.5,0.5);
\draw [line width=1](1.2,-0.5)--(0.5,0.5);
\draw [line width=1](2.5,-0.5)--(0.5,0.5);

\draw [line width=1](0,-0.5)--(2,0.5);
\draw [line width=1](1.2,-0.5)--(2,0.5);
\draw [line width=1](2.5,-0.5)--(2,0.5);

\draw (0,-0.5) node [scale=0.4, circle, draw,fill=black]{};
\draw (0.5,0.5) node [scale=0.4, circle, draw,fill=black]{};
\draw (1.2,-0.5) node [scale=0.4, circle, draw,fill=black]{};
\draw (2,0.5) node [scale=0.4, circle, draw,fill=black]{};
\draw (2.5,-0.5) node [scale=0.4, circle, draw,fill=black]{};

\node [left] at (0.5,0.6){${\small 1}$};
\node [left] at (0,-0.6){${\small 2}$};
\node [right] at (2,0.6){${\small 4}$};
\node [right] at (1.2,-0.6){${\small 3}$};
\node [right] at (2.5,-0.6){${\small 5}$};

\end{tikzpicture}
}\end{minipage}\hspace{0.7cm} (equivalently, $(\{2,5\},\{1,3,4\})=$\hspace{-0.4cm}\begin{minipage}[c]{3.5em}\scalebox{1}{
\begin{tikzpicture}[scale=0.5]

\draw [line width=1](0,-0.5)--(0.5,0.5);
\draw [line width=1](1.2,-0.5)--(0.5,0.5);
\draw [line width=1](2.5,-0.5)--(0.5,0.5);

\draw [line width=1](0,-0.5)--(2,0.5);
\draw [line width=1](1.2,-0.5)--(2,0.5);
\draw [line width=1](2.5,-0.5)--(2,0.5);

\draw (0,-0.5) node [scale=0.4, circle, draw,fill=black]{};
\draw (0.5,0.5) node [scale=0.4, circle, draw,fill=black]{};
\draw (1.2,-0.5) node [scale=0.4, circle, draw,fill=black]{};
\draw (2,0.5) node [scale=0.4, circle, draw,fill=black]{};
\draw (2.5,-0.5) node [scale=0.4, circle, draw,fill=black]{};

\node [left] at (0.5,0.6){${\small 2}$};
\node [left] at (0,-0.6){${\small 1}$};
\node [right] at (2,0.6){${\small 5}$};
\node [right] at (1.2,-0.6){${\small 3}$};
\node [right] at (2.5,-0.6){${\small 4}$};

\end{tikzpicture}
}\end{minipage}\hspace{0.8cm}).

\begin{thm}\label{long-answer} Let $p$ be the POP $(\{1,4\},\{2,3,5\})=$\hspace{-0.4cm}\begin{minipage}[c]{3.5em}\scalebox{1}{
\begin{tikzpicture}[scale=0.5]

\draw [line width=1](0,-0.5)--(0.5,0.5);
\draw [line width=1](1.2,-0.5)--(0.5,0.5);
\draw [line width=1](2.5,-0.5)--(0.5,0.5);

\draw [line width=1](0,-0.5)--(2,0.5);
\draw [line width=1](1.2,-0.5)--(2,0.5);
\draw [line width=1](2.5,-0.5)--(2,0.5);

\draw (0,-0.5) node [scale=0.4, circle, draw,fill=black]{};
\draw (0.5,0.5) node [scale=0.4, circle, draw,fill=black]{};
\draw (1.2,-0.5) node [scale=0.4, circle, draw,fill=black]{};
\draw (2,0.5) node [scale=0.4, circle, draw,fill=black]{};
\draw (2.5,-0.5) node [scale=0.4, circle, draw,fill=black]{};

\node [left] at (0.5,0.6){${\small 1}$};
\node [left] at (0,-0.6){${\small 2}$};
\node [right] at (2,0.6){${\small 4}$};
\node [right] at (1.2,-0.6){${\small 3}$};
\node [right] at (2.5,-0.6){${\small 5}$};

\end{tikzpicture}
}\end{minipage}\hspace{0.8cm}. Then, $$a(n)=\left\{ \begin{array}{ll} 
n! & \mbox{if }n\leq 4\\ 
7a(n-1)-12a(n-2)+4a(n-3)+2b(n-2)& \mbox{if } n\geq 5,\end{array}\right.$$
where $b(1)=0$, $b(2)=1$ and for $n\geq 3$, 
$$b(n)=a(n-2)+b(n-1)+2\sum_{i=2}^{n-2}b(i).$$
\end{thm}

\begin{proof}  For $p$-avoiding permutations $\pi_1\cdots\pi_n$, we let
\begin{itemize}
\item $b(n)$ (resp., $b'(n)$) be the number of $n$-permutations with $\pi_{n-1}=n$ (resp., $\pi_{n-1}=n-1$); 
\item $c(n)$ (resp., $c'(n)$) be the number of $n$-permutations with  $\pi_{i-1}=n-1$ and $\pi_i=n$ (resp., $\pi_{i-1}=n$ and $\pi_i=n-1$) for $2\leq i<n$;
\item $d(n)$ (resp., $d'(n)$) be the number of $n$-permutations with  $\pi_{n-3}=n-1$ and $\pi_{n-1}=n$ (resp., $\pi_{n-3}=n$ and $\pi_{n-1}=n-1$); 
\item $e(n)$ (resp., $e'(n)$) be the number of  $n$-permutations with $\pi_{i-2}=n-1$ and $\pi_i=n$ (resp., $\pi_{i-2}=n$ and $\pi_i=n-1$) for $3\leq i<n-1$; 
\item $f(n)$ (resp., $f'(n)$) be the number of $n$-permutations with $\pi_n\neq n$ (resp., $\pi_n\neq n-1$). 
\end{itemize}

\noindent
Clearly, for $n<5$, $a(n)=n!$, $f(n)=n!-(n-1)!$, $b(n)=(n-1)!$ except for $b(1)=0$. Also, for $n<3$, $c(n)=d(n)=e(n)=0$. Moreover, $p$ is symmetric with respect to $n$ and $n-1$ meaning that swapping $n$ and $n-1$ neither introduces an extra occurrence of $p$ nor removes an existent occurrence of $p$ (both $n$ and $n-1$ can only play the role of $1$ or $4$ in $p$).  Hence, $b'(n)=b(n)$, $c'(n)=c(n)$, $d'(n)=d(n)$, $e'(n)=e(n)$ and $f'(n)=f(n)$, and therefore in what follows we can assume that $n$ is to the right of $n-1$, and at the end we multiply the result by 2.

Since, clearly, there are $a(n-1)$ $n$-permutations $\pi_1\cdots\pi_n$ with $\pi_n=n$, we have $$f(n)=a(n)-a(n-1).$$ If $\pi_i= n$ for $i<n$ then either $\pi_{i-1}=n-1$ or $\pi_{i-2}=n-1$ (or else there is an occurrence of $p$ involving $n-1$ and $n$). In the latter case, we consider separately the case when $\pi_{n-1}=n$. Keeping in mind doubling the result, we obtain
\begin{equation}\label{long-formula-for-an}
a(n)=2(a(n-1)+c(n)+d(n)+e(n)).
\end{equation} 
We now compute $c(n)$. Let $\pi$ be an $n$-permutation counted by $c(n)$. Then, clearly, removing $n$ from $\pi$ gives a $p$-equivalent permutation. Hence, 
\begin{equation}\label{formula-cn}
c(n)=f(n-1)=a(n-1)-a(n-2).
\end{equation} 
Let us compute $e(n)$. If $\pi$ is an $n$-permutation counted by $e(n)$ then to the right of $n$ in $\pi$ there are at least two elements. Consider possible positions $j$ for $n-2$. The case of $j<i-3$ is not possible because then $\pi$ is not $p$-avoiding ($n-2$ and $n$ play the roles of $1$ and $4$ in $p$, respectively). The case of $i+1<j<n$ is not possible either as then $\pi$ is not $p$-avoiding ($n-1$ and $n-2$ play the roles of $1$ and $4$ in $p$, respectively). This leaves us with four possible cases for $j$ to consider.
\begin{itemize}
\item $j=i-3$. In this case, $\pi$ is $p$-equivalent to the permutation $\pi'$ obtained from $\pi$ by removing $n-2$ and reducing $n-1$ and $n$ by 1. Indeed, $n-2$ cannot play the role of $2,3,5$ in an occurrence of $p$ as otherwise $n-1$ and $n$ must play the roles of $1$ and $4$ but there is only one element in $\pi$ between them. On the other hand, if $n-2$ plays the role of $1$ or $4$ in an occurrence of $p$ in $\pi$, then the element $n-2$ in $\pi'$ (the former element $n-1$ in $\pi$) plays the same role in an occurrence of $p$ in $\pi'$. Hence, there are $e(n-1)$ such permutations. 
\item $j=i-1$. In this case, $\pi$ is $p$-equivalent to the permutation $\pi'$ obtained from $\pi$ by removing $n-1$ and $n$ since only one of the consecutive elements $n-1$, $n-2$, $n$ can be involved in an occurrence of $p$. Since in $\pi'$ the maximum element $n-2$ is not in the last two positions, we have $f(n-2)-b(n-2)$  such permutations. 
\item $j=i+1$ or $j=n$. In both cases, $\pi$ is $p$-equivalent to the permutation $\pi'$ obtained from $\pi$ by removing $n-2$ and reducing $n-1$ and $n$ by 1. In $\pi'$, there is one element between the maximum elements $n-2$ and $n-1$, and the element $n-1$ is not the rightmost element. Hence, in each of the two cases, the number of permutations is $d(n-1)+e(n-1)$. 
\end{itemize}  
Summarizing the four cases, we obtain
$$e(n)=3e(n-1)+2d(n-1)+a(n-2)-a(n-3)-b(n-2).$$
Next, we compute $d(n)$. Let $\pi$ be an $n$-permutation counted by $d(n)$. Similarly to our considerations in the case of $e(n)$, we have three possibilities fo position $j$ of the element $n-2$.
\begin{itemize}
\item $j=n-4$. Removing $n-2$ results in a $p$-equivalent permutation. Hence, there are $d(n-1)$ permutations in this case.
\item $j=n-2$. Removing $n-1$ and $n$ results in a $p$-equivalent permutation. Hence, there are $b(n-2)$ permutations in this case.
\item $j=n$. Removing $n-2$ and $n$ results in a $p$-equivalent permutation. Hence, there are $b(n-2)$ permutations in this case.
\end{itemize} 
Summarizing the three cases, we obtain
$$d(n)=2b(n-2)+d(n-1).$$
Finally, we compute $b(n)$. Let $\pi$ be an $n$-permutation counted by $b(n)$ and $\pi_j=n-1$. Clearly, to avoid $p$ we must have $j>n-4$. If $j=n$ then $\pi$ is $p$-equivalent to the permutation obtained by removing $n-1$ and $n$ from $\pi$, and there are $a(n-2)$ such permutations. If $j=n-2$ then $\pi$ is $p$-equivalent to the permutation obtained by removing $n$ from $\pi$, and there are $b(n-1)$ such permutations. Finally, if $j=n-3$ then by definition, there are $d(n)$ such permutations. Hence,
$$b(n)=a(n-2)+b(n-1)+d(n).$$
As the final step, we simplify the obtained formulas.
$$d(n)=2b(n-2)+d(n-1) = 2b(n-2)+2b(n-3)+d(n-2) = \cdots = 2\sum_{i=2}^{n-2} b(i)$$
since $b(1)=0$. But then
$$b(n)=a(n-2)+b(n-1)+d(n) = a(n-2)+b(n-1)+ 2\sum_{i=2}^{n-2} b(i).$$
From (\ref{long-formula-for-an}) and (\ref{formula-cn}), $a(n)=2[2a(n-1)-a(n-2)+d(n)+e(n)]$ and hence $d(n)+e(n)=a(n)/2-2a(n-1)+a(n-2)$. On the other hand, 
$$d(n)+e(n)= 2b(n-2)+d(n-1) + 3e(n-1)+2d(n-1)+a(n-2)-a(n-3)-b(n-2) $$
$$= b(n-2)+a(n-2)-a(n-3)+3(d(n-1)+e(n-1))$$
$$= b(n-2)+a(n-2)-a(n-3)+3(a(n-1)/2-2a(n-2)+a(n-3))$$ 
$$=  b(n-2)+3a(n-1)/2-5a(n-2)+2a(n-3).$$ Inserting the last expression for $d(n)+e(n)$ into (\ref{long-formula-for-an}), we obtain the desired result:
$$a(n) = 2(2a(n-1)-a(n-2)+(b(n-2)+3a(n-1)/2-5a(n-2)+2a(n-3)))$$ 
$$= 7a(n-1)-12a(n-2)+4a(n-3)+2b(n-2).$$
\end{proof}

\section{Wilf-classification for N-pattern avoidance}\label{WC-N-patterns-sec}

Note that the minimal connected bipartite graph that is not a complete bipartite graph is a $4$-path. Due to Theorem~\ref{trivial-sym-thm}, we may assume that the
corresponding poset defining a bipartite POP of length 4 has the N-shape, and we call such POPs {\em N-patterns}. The following two theorems were obtained in \cite{GK19}.

\begin{thm}[\cite{GK19}]\label{thm-12} 
For the N-pattern \hspace{-3.5mm}
\begin{minipage}[c]{3.8em}\scalebox{1}{
\begin{tikzpicture}[scale=0.3]
\draw [line width=1](0,0)--(0,1)--(1,0)--(1,1);
\draw (0,0) node [scale=0.3, circle, draw,fill=black]{};
\draw (0,1) node [scale=0.3, circle, draw,fill=black]{};
\draw (1,0) node [scale=0.3, circle, draw,fill=black]{};
\draw (1,1) node [scale=0.3, circle, draw,fill=black]{};
\node [left] at (0,-0.1){\small$2$};
\node [left] at (0,1.1){\small$1$};
\node [right] at (1,-0.1){\small$3$};
\node [right] at (1,1.1){\small$4$};
\end{tikzpicture}
}\end{minipage}
\hspace{-0.4cm}, $a(0)=a(1)=1$ and, for $n\geq 2$, $a(n) = 4a(n-1) - 3a(n-2) + 1$, so that 
\begin{equation}\label{thm-12-formula}
a(n)=\frac{3^n-2n+3}{4}.\nonumber
\end{equation}
Also, $$\sum_{n\geq 0}a(n)x^n=\frac{(1-2x)^2}{(1-3x)(1-x)^2}.$$
\end{thm}

\begin{thm}[\cite{GK19}]\label{thm-11} 
For the N-pattern $p=$ \hspace{-3.5mm}
\begin{minipage}[c]{3.8em}\scalebox{1}{
\begin{tikzpicture}[scale=0.3]
\draw [line width=1](0,0)--(0,1)--(1,0)--(1,1);
\draw (0,0) node [scale=0.3, circle, draw,fill=black]{};
\draw (0,1) node [scale=0.3, circle, draw,fill=black]{};
\draw (1,0) node [scale=0.3, circle, draw,fill=black]{};
\draw (1,1) node [scale=0.3, circle, draw,fill=black]{};
\node [left] at (0,-0.1){\small$3$};
\node [left] at (0,1.1){\small$1$};
\node [right] at (1,-0.1){\small$2$};
\node [right] at (1,1.1){\small$4$};
\end{tikzpicture}
}\end{minipage}
we have $a(0)=a(1)=1$ and, for $n\geq 2$, $a(n) = 4a(n-1) - 3a(n-2) + a(n-3)$, so that, for $n\geq 1$, 
\begin{equation}\label{thm-11-formula}
a(n)=\sum_{i=0}^{n-1} {n+2i-1\choose 3i}.\nonumber
\end{equation}
Also, $$\sum_{n\geq 0}a(n)x^n= \frac{1-3x+x^2}{1-4x+3x^2-x^3}.$$
\end{thm}

Reading the labels of an N-pattern through the path starting from the left lower vertex to the right upper vertex, we can encode such a pattern by a 4-permutation, called {\em N-word}, that we record in bold to distinguish from a classical pattern. For example, the N-pattern \hspace{-3.5mm}
\begin{minipage}[c]{3.8em}\scalebox{1}{
\begin{tikzpicture}[scale=0.3]
\draw [line width=1](0,0)--(0,1)--(1,0)--(1,1);
\draw (0,0) node [scale=0.3, circle, draw,fill=black]{};
\draw (0,1) node [scale=0.3, circle, draw,fill=black]{};
\draw (1,0) node [scale=0.3, circle, draw,fill=black]{};
\draw (1,1) node [scale=0.3, circle, draw,fill=black]{};
\node [left] at (0,-0.1){\small$3$};
\node [left] at (0,1.1){\small$1$};
\node [right] at (1,-0.1){\small$2$};
\node [right] at (1,1.1){\small$4$};
\end{tikzpicture}
}\end{minipage} is encoded by the N-word {\bf 3124}. In this section, for the reader's convenience we indicate elements of posets in bold in order to distinguish them from elements of permutations.

\begin{table}
\begin{tabular}{c||c}
{\bf N-patterns} & {\bf Reference} \\
\hline
\multicolumn{2}{l}{{\bf Wilf-class 1} given by the recurence $a(n)=4a(n-1)-3a(n-2)+1$.} \\
\multicolumn{2}{l}{This is the sequence $A111277$ in \cite{oeis} ($1, 2, 6, 19, 59, 180, 544, 1637,4917,\ldots$).}\\
\hline
\hspace{-3.5mm}
\begin{minipage}[c]{3.8em}\scalebox{1}{
\begin{tikzpicture}[scale=0.3]
\draw [line width=1](0,0)--(0,1)--(1,0)--(1,1);
\draw (0,0) node [scale=0.3, circle, draw,fill=black]{};
\draw (0,1) node [scale=0.3, circle, draw,fill=black]{};
\draw (1,0) node [scale=0.3, circle, draw,fill=black]{};
\draw (1,1) node [scale=0.3, circle, draw,fill=black]{};
\node [left] at (0,-0.1){\small$2$};
\node [left] at (0,1.1){\small$1$};
\node [right] at (1,-0.1){\small$3$};
\node [right] at (1,1.1){\small$4$};
\end{tikzpicture}
}\end{minipage}
\hspace{-0.4cm},  \hspace{-3.5mm}
\begin{minipage}[c]{3.8em}\scalebox{1}{
\begin{tikzpicture}[scale=0.3]
\draw [line width=1](0,0)--(0,1)--(1,0)--(1,1);
\draw (0,0) node [scale=0.3, circle, draw,fill=black]{};
\draw (0,1) node [scale=0.3, circle, draw,fill=black]{};
\draw (1,0) node [scale=0.3, circle, draw,fill=black]{};
\draw (1,1) node [scale=0.3, circle, draw,fill=black]{};
\node [left] at (0,-0.1){\small$3$};
\node [left] at (0,1.1){\small$4$};
\node [right] at (1,-0.1){\small$2$};
\node [right] at (1,1.1){\small$1$};
\end{tikzpicture}
}\end{minipage}
\hspace{-0.4cm}, \hspace{-3.5mm}
\begin{minipage}[c]{3.8em}\scalebox{1}{
\begin{tikzpicture}[scale=0.3]
\draw [line width=1](0,0)--(0,1)--(1,0)--(1,1);
\draw (0,0) node [scale=0.3, circle, draw,fill=black]{};
\draw (0,1) node [scale=0.3, circle, draw,fill=black]{};
\draw (1,0) node [scale=0.3, circle, draw,fill=black]{};
\draw (1,1) node [scale=0.3, circle, draw,fill=black]{};
\node [left] at (0,-0.1){\small$1$};
\node [left] at (0,1.1){\small$2$};
\node [right] at (1,-0.1){\small$4$};
\node [right] at (1,1.1){\small$3$};
\end{tikzpicture}
}\end{minipage}
\hspace{-0.4cm}, \hspace{-3.5mm}
\begin{minipage}[c]{3.8em}\scalebox{1}{
\begin{tikzpicture}[scale=0.3]
\draw [line width=1](0,0)--(0,1)--(1,0)--(1,1);
\draw (0,0) node [scale=0.3, circle, draw,fill=black]{};
\draw (0,1) node [scale=0.3, circle, draw,fill=black]{};
\draw (1,0) node [scale=0.3, circle, draw,fill=black]{};
\draw (1,1) node [scale=0.3, circle, draw,fill=black]{};
\node [left] at (0,-0.1){\small$4$};
\node [left] at (0,1.1){\small$3$};
\node [right] at (1,-0.1){\small$1$};
\node [right] at (1,1.1){\small$2$};
\end{tikzpicture}
}\end{minipage}
\hspace{-0.4cm} & Theorem~\ref{thm-12}\\
\hline
\hspace{-3.5mm}
\begin{minipage}[c]{3.8em}\scalebox{1}{
\begin{tikzpicture}[scale=0.3]
\draw [line width=1](0,0)--(0,1)--(1,0)--(1,1);
\draw (0,0) node [scale=0.3, circle, draw,fill=black]{};
\draw (0,1) node [scale=0.3, circle, draw,fill=black]{};
\draw (1,0) node [scale=0.3, circle, draw,fill=black]{};
\draw (1,1) node [scale=0.3, circle, draw,fill=black]{};
\node [left] at (0,-0.1){\small$2$};
\node [left] at (0,1.1){\small$1$};
\node [right] at (1,-0.1){\small$4$};
\node [right] at (1,1.1){\small$3$};
\end{tikzpicture}
}\end{minipage}
\hspace{-0.4cm},  \hspace{-3.5mm}
\begin{minipage}[c]{3.8em}\scalebox{1}{
\begin{tikzpicture}[scale=0.3]
\draw [line width=1](0,0)--(0,1)--(1,0)--(1,1);
\draw (0,0) node [scale=0.3, circle, draw,fill=black]{};
\draw (0,1) node [scale=0.3, circle, draw,fill=black]{};
\draw (1,0) node [scale=0.3, circle, draw,fill=black]{};
\draw (1,1) node [scale=0.3, circle, draw,fill=black]{};
\node [left] at (0,-0.1){\small$3$};
\node [left] at (0,1.1){\small$4$};
\node [right] at (1,-0.1){\small$1$};
\node [right] at (1,1.1){\small$2$};
\end{tikzpicture}
}\end{minipage}
\hspace{-0.4cm} &  Theorem~\ref{thm-two-N-patterns} \\
\hline
\hspace{-3.5mm}
\begin{minipage}[c]{3.8em}\scalebox{1}{
\begin{tikzpicture}[scale=0.3]
\draw [line width=1](0,0)--(0,1)--(1,0)--(1,1);
\draw (0,0) node [scale=0.3, circle, draw,fill=black]{};
\draw (0,1) node [scale=0.3, circle, draw,fill=black]{};
\draw (1,0) node [scale=0.3, circle, draw,fill=black]{};
\draw (1,1) node [scale=0.3, circle, draw,fill=black]{};
\node [left] at (0,-0.1){\small$1$};
\node [left] at (0,1.1){\small$2$};
\node [right] at (1,-0.1){\small$3$};
\node [right] at (1,1.1){\small$4$};
\end{tikzpicture}
}\end{minipage}
\hspace{-0.4cm},  \hspace{-3.5mm}
\begin{minipage}[c]{3.8em}\scalebox{1}{
\begin{tikzpicture}[scale=0.3]
\draw [line width=1](0,0)--(0,1)--(1,0)--(1,1);
\draw (0,0) node [scale=0.3, circle, draw,fill=black]{};
\draw (0,1) node [scale=0.3, circle, draw,fill=black]{};
\draw (1,0) node [scale=0.3, circle, draw,fill=black]{};
\draw (1,1) node [scale=0.3, circle, draw,fill=black]{};
\node [left] at (0,-0.1){\small$4$};
\node [left] at (0,1.1){\small$3$};
\node [right] at (1,-0.1){\small$2$};
\node [right] at (1,1.1){\small$1$};
\end{tikzpicture}
}\end{minipage}
\hspace{-0.4cm} & Theorem~\ref{thm-two-N-patterns-2}\\
\hline
\hspace{-3.5mm}
\begin{minipage}[c]{3.8em}\scalebox{1}{
\begin{tikzpicture}[scale=0.3]
\draw [line width=1](0,0)--(0,1)--(1,0)--(1,1);
\draw (0,0) node [scale=0.3, circle, draw,fill=black]{};
\draw (0,1) node [scale=0.3, circle, draw,fill=black]{};
\draw (1,0) node [scale=0.3, circle, draw,fill=black]{};
\draw (1,1) node [scale=0.3, circle, draw,fill=black]{};
\node [left] at (0,-0.1){\small$1$};
\node [left] at (0,1.1){\small$3$};
\node [right] at (1,-0.1){\small$2$};
\node [right] at (1,1.1){\small$4$};
\end{tikzpicture}
}\end{minipage}
\hspace{-0.4cm},  \hspace{-3.5mm}
\begin{minipage}[c]{3.8em}\scalebox{1}{
\begin{tikzpicture}[scale=0.3]
\draw [line width=1](0,0)--(0,1)--(1,0)--(1,1);
\draw (0,0) node [scale=0.3, circle, draw,fill=black]{};
\draw (0,1) node [scale=0.3, circle, draw,fill=black]{};
\draw (1,0) node [scale=0.3, circle, draw,fill=black]{};
\draw (1,1) node [scale=0.3, circle, draw,fill=black]{};
\node [left] at (0,-0.1){\small$4$};
\node [left] at (0,1.1){\small$2$};
\node [right] at (1,-0.1){\small$3$};
\node [right] at (1,1.1){\small$1$};
\end{tikzpicture}
}\end{minipage}
\hspace{-0.4cm} & Theorem~\ref{thm-N-pattern-1111}\\
\hline
\hspace{-3.5mm}
\begin{minipage}[c]{3.8em}\scalebox{1}{
\begin{tikzpicture}[scale=0.3]
\draw [line width=1](0,0)--(0,1)--(1,0)--(1,1);
\draw (0,0) node [scale=0.3, circle, draw,fill=black]{};
\draw (0,1) node [scale=0.3, circle, draw,fill=black]{};
\draw (1,0) node [scale=0.3, circle, draw,fill=black]{};
\draw (1,1) node [scale=0.3, circle, draw,fill=black]{};
\node [left] at (0,-0.1){\small$1$};
\node [left] at (0,1.1){\small$4$};
\node [right] at (1,-0.1){\small$3$};
\node [right] at (1,1.1){\small$2$};
\end{tikzpicture}
}\end{minipage}
\hspace{-0.4cm},  \hspace{-3.5mm}
\begin{minipage}[c]{3.8em}\scalebox{1}{
\begin{tikzpicture}[scale=0.3]
\draw [line width=1](0,0)--(0,1)--(1,0)--(1,1);
\draw (0,0) node [scale=0.3, circle, draw,fill=black]{};
\draw (0,1) node [scale=0.3, circle, draw,fill=black]{};
\draw (1,0) node [scale=0.3, circle, draw,fill=black]{};
\draw (1,1) node [scale=0.3, circle, draw,fill=black]{};
\node [left] at (0,-0.1){\small$4$};
\node [left] at (0,1.1){\small$1$};
\node [right] at (1,-0.1){\small$2$};
\node [right] at (1,1.1){\small$3$};
\end{tikzpicture}
}\end{minipage}
\hspace{-0.4cm}, \hspace{-3.5mm}
\begin{minipage}[c]{3.8em}\scalebox{1}{
\begin{tikzpicture}[scale=0.3]
\draw [line width=1](0,0)--(0,1)--(1,0)--(1,1);
\draw (0,0) node [scale=0.3, circle, draw,fill=black]{};
\draw (0,1) node [scale=0.3, circle, draw,fill=black]{};
\draw (1,0) node [scale=0.3, circle, draw,fill=black]{};
\draw (1,1) node [scale=0.3, circle, draw,fill=black]{};
\node [left] at (0,-0.1){\small$3$};
\node [left] at (0,1.1){\small$2$};
\node [right] at (1,-0.1){\small$1$};
\node [right] at (1,1.1){\small$4$};
\end{tikzpicture}
}\end{minipage}
\hspace{-0.4cm}, \hspace{-3.5mm}
\begin{minipage}[c]{3.8em}\scalebox{1}{
\begin{tikzpicture}[scale=0.3]
\draw [line width=1](0,0)--(0,1)--(1,0)--(1,1);
\draw (0,0) node [scale=0.3, circle, draw,fill=black]{};
\draw (0,1) node [scale=0.3, circle, draw,fill=black]{};
\draw (1,0) node [scale=0.3, circle, draw,fill=black]{};
\draw (1,1) node [scale=0.3, circle, draw,fill=black]{};
\node [left] at (0,-0.1){\small$2$};
\node [left] at (0,1.1){\small$3$};
\node [right] at (1,-0.1){\small$4$};
\node [right] at (1,1.1){\small$1$};
\end{tikzpicture}
}\end{minipage}
\hspace{-0.4cm} & Theorem~\ref{former-conj-1}\\
\hline
\multicolumn{2}{l}{{\bf Wilf-class 2} given by the recurrence $a(n)=4a(n-1) - 3a(n-2) + a(n-3)$.} \\
\multicolumn{2}{l}{This is the sequence $A052544$ in \cite{oeis} ($1, 2, 6, 19, 60, 189, 595, 1873, 5896,\ldots$).}\\
\hline
 \hspace{-3.5mm}
\begin{minipage}[c]{3.8em}\scalebox{1}{
\begin{tikzpicture}[scale=0.3]
\draw [line width=1](0,0)--(0,1)--(1,0)--(1,1);
\draw (0,0) node [scale=0.3, circle, draw,fill=black]{};
\draw (0,1) node [scale=0.3, circle, draw,fill=black]{};
\draw (1,0) node [scale=0.3, circle, draw,fill=black]{};
\draw (1,1) node [scale=0.3, circle, draw,fill=black]{};
\node [left] at (0,-0.1){\small$3$};
\node [left] at (0,1.1){\small$1$};
\node [right] at (1,-0.1){\small$2$};
\node [right] at (1,1.1){\small$4$};
\end{tikzpicture}
}\end{minipage}
\hspace{-0.4cm},  \hspace{-3.5mm}
\begin{minipage}[c]{3.8em}\scalebox{1}{
\begin{tikzpicture}[scale=0.3]
\draw [line width=1](0,0)--(0,1)--(1,0)--(1,1);
\draw (0,0) node [scale=0.3, circle, draw,fill=black]{};
\draw (0,1) node [scale=0.3, circle, draw,fill=black]{};
\draw (1,0) node [scale=0.3, circle, draw,fill=black]{};
\draw (1,1) node [scale=0.3, circle, draw,fill=black]{};
\node [left] at (0,-0.1){\small$2$};
\node [left] at (0,1.1){\small$4$};
\node [right] at (1,-0.1){\small$3$};
\node [right] at (1,1.1){\small$1$};
\end{tikzpicture}
}\end{minipage}
\hspace{-0.4cm}, \hspace{-3.5mm}
\begin{minipage}[c]{3.8em}\scalebox{1}{
\begin{tikzpicture}[scale=0.3]
\draw [line width=1](0,0)--(0,1)--(1,0)--(1,1);
\draw (0,0) node [scale=0.3, circle, draw,fill=black]{};
\draw (0,1) node [scale=0.3, circle, draw,fill=black]{};
\draw (1,0) node [scale=0.3, circle, draw,fill=black]{};
\draw (1,1) node [scale=0.3, circle, draw,fill=black]{};
\node [left] at (0,-0.1){\small$1$};
\node [left] at (0,1.1){\small$3$};
\node [right] at (1,-0.1){\small$4$};
\node [right] at (1,1.1){\small$2$};
\end{tikzpicture}
}\end{minipage}
\hspace{-0.4cm}, \hspace{-3.5mm}
\begin{minipage}[c]{3.8em}\scalebox{1}{
\begin{tikzpicture}[scale=0.3]
\draw [line width=1](0,0)--(0,1)--(1,0)--(1,1);
\draw (0,0) node [scale=0.3, circle, draw,fill=black]{};
\draw (0,1) node [scale=0.3, circle, draw,fill=black]{};
\draw (1,0) node [scale=0.3, circle, draw,fill=black]{};
\draw (1,1) node [scale=0.3, circle, draw,fill=black]{};
\node [left] at (0,-0.1){\small$4$};
\node [left] at (0,1.1){\small$2$};
\node [right] at (1,-0.1){\small$1$};
\node [right] at (1,1.1){\small$3$};
\end{tikzpicture}
}\end{minipage}
\hspace{-0.4cm} & Theorem~\ref{thm-11}\\
\hline
 \hspace{-3.5mm}
\begin{minipage}[c]{3.8em}\scalebox{1}{
\begin{tikzpicture}[scale=0.3]
\draw [line width=1](0,0)--(0,1)--(1,0)--(1,1);
\draw (0,0) node [scale=0.3, circle, draw,fill=black]{};
\draw (0,1) node [scale=0.3, circle, draw,fill=black]{};
\draw (1,0) node [scale=0.3, circle, draw,fill=black]{};
\draw (1,1) node [scale=0.3, circle, draw,fill=black]{};
\node [left] at (0,-0.1){\small$1$};
\node [left] at (0,1.1){\small$4$};
\node [right] at (1,-0.1){\small$2$};
\node [right] at (1,1.1){\small$3$};
\end{tikzpicture}
}\end{minipage}
\hspace{-0.4cm},  \hspace{-3.5mm}
\begin{minipage}[c]{3.8em}\scalebox{1}{
\begin{tikzpicture}[scale=0.3]
\draw [line width=1](0,0)--(0,1)--(1,0)--(1,1);
\draw (0,0) node [scale=0.3, circle, draw,fill=black]{};
\draw (0,1) node [scale=0.3, circle, draw,fill=black]{};
\draw (1,0) node [scale=0.3, circle, draw,fill=black]{};
\draw (1,1) node [scale=0.3, circle, draw,fill=black]{};
\node [left] at (0,-0.1){\small$4$};
\node [left] at (0,1.1){\small$1$};
\node [right] at (1,-0.1){\small$3$};
\node [right] at (1,1.1){\small$2$};
\end{tikzpicture}
}\end{minipage}
\hspace{-0.4cm}, \hspace{-3.5mm}
\begin{minipage}[c]{3.8em}\scalebox{1}{
\begin{tikzpicture}[scale=0.3]
\draw [line width=1](0,0)--(0,1)--(1,0)--(1,1);
\draw (0,0) node [scale=0.3, circle, draw,fill=black]{};
\draw (0,1) node [scale=0.3, circle, draw,fill=black]{};
\draw (1,0) node [scale=0.3, circle, draw,fill=black]{};
\draw (1,1) node [scale=0.3, circle, draw,fill=black]{};
\node [left] at (0,-0.1){\small$2$};
\node [left] at (0,1.1){\small$3$};
\node [right] at (1,-0.1){\small$1$};
\node [right] at (1,1.1){\small$4$};
\end{tikzpicture}
}\end{minipage}
\hspace{-0.4cm}, \hspace{-3.5mm}
\begin{minipage}[c]{3.8em}\scalebox{1}{
\begin{tikzpicture}[scale=0.3]
\draw [line width=1](0,0)--(0,1)--(1,0)--(1,1);
\draw (0,0) node [scale=0.3, circle, draw,fill=black]{};
\draw (0,1) node [scale=0.3, circle, draw,fill=black]{};
\draw (1,0) node [scale=0.3, circle, draw,fill=black]{};
\draw (1,1) node [scale=0.3, circle, draw,fill=black]{};
\node [left] at (0,-0.1){\small$3$};
\node [left] at (0,1.1){\small$2$};
\node [right] at (1,-0.1){\small$4$};
\node [right] at (1,1.1){\small$1$};
\end{tikzpicture}
}\end{minipage}
\hspace{-0.4cm} & Theorem~\ref{thm-1423} \\
\hline
\multicolumn{2}{l}{{\bf Wilf-class 3} given by the recurrence $a(n) = 3a(n-1) + a(n-2) - a(n-3)$.} \\
\multicolumn{2}{l}{The sequence begins with $1, 2, 6, 19, 61, 196, 630, 2025,6509,\ldots$.}\\
\hline
 \hspace{-3.5mm}
\begin{minipage}[c]{3.8em}\scalebox{1}{
\begin{tikzpicture}[scale=0.3]
\draw [line width=1](0,0)--(0,1)--(1,0)--(1,1);
\draw (0,0) node [scale=0.3, circle, draw,fill=black]{};
\draw (0,1) node [scale=0.3, circle, draw,fill=black]{};
\draw (1,0) node [scale=0.3, circle, draw,fill=black]{};
\draw (1,1) node [scale=0.3, circle, draw,fill=black]{};
\node [left] at (0,-0.1){\small$3$};
\node [left] at (0,1.1){\small$1$};
\node [right] at (1,-0.1){\small$4$};
\node [right] at (1,1.1){\small$2$};
\end{tikzpicture}
}\end{minipage}
\hspace{-0.4cm},  \hspace{-3.5mm}
\begin{minipage}[c]{3.8em}\scalebox{1}{
\begin{tikzpicture}[scale=0.3]
\draw [line width=1](0,0)--(0,1)--(1,0)--(1,1);
\draw (0,0) node [scale=0.3, circle, draw,fill=black]{};
\draw (0,1) node [scale=0.3, circle, draw,fill=black]{};
\draw (1,0) node [scale=0.3, circle, draw,fill=black]{};
\draw (1,1) node [scale=0.3, circle, draw,fill=black]{};
\node [left] at (0,-0.1){\small$2$};
\node [left] at (0,1.1){\small$4$};
\node [right] at (1,-0.1){\small$1$};
\node [right] at (1,1.1){\small$3$};
\end{tikzpicture}
}\end{minipage}
\hspace{-0.4cm} & Theorem~\ref{thm-new-wilf-equiv} \\
\end{tabular}
\caption{Wilf-equivalence classification for N-patterns where the patterns in the same row are Wilf-equivalent by Theorem~\ref{trivial-sym-thm}}\label{N-patterns-tab}
\end{table}

There are 24 N-words. By Theorem~\ref{trivial-sym-thm}, complementing the labels of a poset gives a Wilf-equivalent pattern, so in the cases to consider we can assume that in an N-word {\bf 1} is to the left of {\bf 4}. Moreover, by Theorem~\ref{trivial-sym-thm}, flipping the poset with respect to a horizontal line then complementing the labels and reading the word backward (to obtain N-shape with {\bf 1}  being to the left of {\bf 4}) we see that {\bf 1423} is equivalent to {\bf 2314}, {\bf 1432} is equivalent to {\bf 3214}, {\bf 1342} is equivalent to {\bf 3124}, {\bf 1243} is equivalent to {\bf 2134}, while {\bf 1234}, {\bf 1324}, {\bf 2143} and {\bf 3142} are equivalent to themselves. These observations leave us with eight cases to consider, two of which are already considered in Theorems~\ref{thm-12} and~\ref{thm-11}. It turns out that there are three Wilf-equivalence classes for the 24 N-patterns, which are summarized in Table~\ref{N-patterns-tab}.

\subsection{N-patterns in Wilf-equivalence class 1 in Table~\ref{N-patterns-tab}}\label{Wilf-class-1-subsec}

\begin{thm}\label{thm-two-N-patterns} For the N-bipartite POP $p=$\hspace{-3.5mm}
\begin{minipage}[c]{3.8em}\scalebox{1}{
\begin{tikzpicture}[scale=0.3]
\draw [line width=1](0,0)--(0,1)--(1,0)--(1,1);
\draw (0,0) node [scale=0.3, circle, draw,fill=black]{};
\draw (0,1) node [scale=0.3, circle, draw,fill=black]{};
\draw (1,0) node [scale=0.3, circle, draw,fill=black]{};
\draw (1,1) node [scale=0.3, circle, draw,fill=black]{};
\node [left] at (0,-0.1){\small$2$};
\node [left] at (0,1.1){\small$1$};
\node [right] at (1,-0.1){\small$4$};
\node [right] at (1,1.1){\small$3$};
\end{tikzpicture}
}\end{minipage}  \hspace{-0.4cm},  we have $a(0)=a(1)=1$ and, for $n\geq 2$, $a(n) = 4a(n-1) - 3a(n-2) + 1$, hence giving the formula for $a(n)$ and the g.f.\ in the statement of Theorem~\ref{thm-12}.
\end{thm}

\begin{proof} 
Let $b(n)$ (resp., $c(n)$) be the number of $p$-avoiding $n$-permutations where $n$ is not the rightmost element (resp., and the elements to the right of $n$ are in increasing order). Suppose that $\pi_1\cdots\pi_n\in S_n$. Clearly, if $\pi_n=n$ (or, $\pi_n=n-1$) then $\pi_n$ cannot be involved in an occurrence of $p$. Hence, removing $\pi_n$ one obtains a $p$-equivalent permutation and therefore $b(n)=a(n)-a(n-1)$.

Let $\pi=\pi_1\cdots\pi_n$ be $p$-avoiding. One can see that there are four possibilities to avoid $p$.
\begin{itemize}
\item $\pi_n=n$. Clearly, there are $a(n-1)$ such permutations.
\item $\pi_n=n-1$. Clearly, there are $a(n-1)$ such permutations.
\item $\pi_{i-1}=n-1$ and $\pi_i=n$ for $2\leq i<n$. We have $b(n-1)=a(n-1)-a(n-2)$ such permutations since removing $n$ from $\pi$ gives a $p$-equivalent permutation.
\item $\pi_{i-1}=n$ and $\pi_i=n-1$ for $2\leq i<n$. In this case, $\pi_{i+1}\cdots\pi_n$ must be an increasing permutation or else, there is an occurrence of $p$ involving the elements $n-1$ and $n$. There are $c(n-1)$ permutations in this case since removing $n$ we obtain a permutation $\pi'$ counted by $c(n-1)$ and $n$ can be inserted back in any such $\pi'$ next to the left of $n-1$ without introducing an occurrence of $p$.
\end{itemize} 
Summarising all the cases above, we obtain
\begin{equation}\label{a-recurssion-aac}
a(n)=3a(n-1)-a(n-2)+c(n-1).
\end{equation}
Next, we find a recurrence relation for $c(n)$. Supposed  that $\pi=\pi_1\cdots\pi_n$ is an $n$-permutation counted by $c(n)$. We consider all possibilities for position of $n-1$ in $\pi$.
\begin{itemize}
\item $n-1$ is to the left of $n$, in which case $n-1$ must be immediately to the left of $n$ to avoid $p$. But then, removing $n$ from $\pi$ we obtain a $p$-equivalent permutation counted by $c(n-1)$. So, there are $c(n-1)$ permutations in this case.  
\item $\pi_{n-1}\pi_n=n(n-1)$. Removing $n(n-1)$ from $\pi$ we obtain a $p$-avoiding $(n-2)$-permutation, and conversely, we can adjoin $n(n-1)$ to the right of any permutation counted by $a(n-2)$ without creating an occurrence of $p$. Hence, we have $a(n-2)$ permutations in this case. 
\item $\pi_i\cdots\pi_n=n\pi_{i+1}\cdots \pi_{n-1}(n-1)$ where $i<n-1$ and $\pi_{i+1}\cdots \pi_{n-1}$ is an increasing permutation. Removing $(n-1)$ we obtain a permutation counted by $c(n-1)$, and conversely, we can adjoin the element $n-1$ at the end of any permutation $\pi'$ counted by $c(n-1)$ (and replacing $n-1$ in $\pi'$ by $n$) without introducing an occurrence of $p$. Hence, we have $c(n-1)$ permutations in this case. 
\end{itemize}
So, $c(n)=2c(n-1)+a(n-2)$.  Now, we will prove by induction by $n$ that $c(n)=a(n)-2a(n-1)+1$ for $n\geq 1$. The base cases of $c(1)=0$ and $c(2)=1$ are easy to check. Suppose $c(n-1)=a(n-1)-2a(n-2)+1$, then
$$c(n)=2c(n-1)+a(n-2)=c(n-1)+a(n-1)-a(n-2)+1=a(n)-2a(n-1)+1$$ as desired, where the last equality is obtained by using (\ref{a-recurssion-aac}). Finally,
$$a(n)=3a(n-1)-a(n-2)+c(n-1)=4a(n-1)-3a(n-2)+1.$$
\end{proof} 

\begin{thm}\label{thm-two-N-patterns-2} For the N-bipartite POP $p=$\hspace{-3.5mm}
\begin{minipage}[c]{3.8em}\scalebox{1}{
\begin{tikzpicture}[scale=0.3]
\draw [line width=1](0,0)--(0,1)--(1,0)--(1,1);
\draw (0,0) node [scale=0.3, circle, draw,fill=black]{};
\draw (0,1) node [scale=0.3, circle, draw,fill=black]{};
\draw (1,0) node [scale=0.3, circle, draw,fill=black]{};
\draw (1,1) node [scale=0.3, circle, draw,fill=black]{};
\node [left] at (0,-0.1){\small$1$};
\node [left] at (0,1.1){\small$2$};
\node [right] at (1,-0.1){\small$3$};
\node [right] at (1,1.1){\small$4$};
\end{tikzpicture}
}\end{minipage} \hspace{-0.5cm} , we have $a(0)=a(1)=1$ and, for $n\geq 2$, $a(n) = 4a(n-1) - 3a(n-2) + 1$, hence giving the formula for $a(n)$ and the g.f.\ in the statement of Theorem~\ref{thm-12}.
\end{thm}

\begin{proof} Let $\pi=\pi_1\cdots\pi_n$ be a $p$-avoiding permutation. If $\pi_1=n$ then there are clearly $a(n-1)$ such permutations. Let $b(n)$ be the number of $p$-avoiding $n$-permutations $\pi_1\cdots\pi_n$ with $\pi_i=n$ for $1<i\leq n$. So,
\begin{equation}\label{a=a+b-relation}
a(n)=a(n-1)+b(n).
\end{equation} 
Note that in any permutation counted by $b(n)$, $\pi_{i+1}>\pi_{i+2}>\cdots >\pi_n$ because if $\pi_j<\pi_s$ for some $i<j<s\leq n$ then the elements $\pi_1$, $n$, $\pi_j$ and $\pi_s$ form an occurrence of $p$. We claim that
\begin{equation}\label{b=3b}
b(n)=3b(n-1)+1.
\end{equation}
Indeed, permutations counted by $b(n)$ fall in one of the following four cases, where we noted that if the element $n-1$ is to the left of $n$ then either the elements are next to each other, or $\pi_1=n-1$ (otherwise, an occurrence of $p$ involving $n$ and $n-1$ would be created):
\begin{itemize}
\item The element $n-1$ is to the right of $n$, and hence $\pi_{i+1}=n-1$. Removing $n$ we obtain a permutation counted by $b(n-1)$. Vice versa, any permutation counted by $b(n-1)$ can be extended to a permutation counted by $b(n)$ by inserting $n$ immediately to the left of the element $n-1$. Thus, we have $b(n-1)$ permutations in this case.
\item $\pi_1=n-1$ and $\pi_2=n$. There is one such permutation. 
\item $\pi_1=n-1$, $\pi_2\neq n$. Removing $n-1$ and replacing $n$ by $n-1$ in the obtained permutation, we get a permutation counted by $b(n-1)$. This operation is invertible as $n-1$ in the leftmost position of an $n$-permutation cannot be involved in an occurrence of $p$. Hence, we have $b(n-1)$ permutations in this case.
\item $\pi_{i-1}\pi_i=(n-1)n$ for $i>2$. It is easy to check that removing $n$ we obtain a $p$-equivalent permutation counted by $b(n-1)$. Thus, we have $b(n-1)$ permutations in this case completing our proof of (\ref{b=3b}). 
\end{itemize}

From (\ref{a=a+b-relation}) and (\ref{b=3b}), to complete the proof of the theorem, clearly we only need to show that, for $n\geq 2$,
$$b(n)=3a(n-1)-3a(n-2)+1.$$
Indeed, from (\ref{a=a+b-relation}), $3a(n-1)-3a(n-2)=3b(n-1)$, and hence
$$3a(n-1)-3a(n-2)+1=3b(n-1)+1=b(n)$$
as desired.
\end{proof} 

\begin{thm}\label{thm-N-pattern-1111} For the POP $p=$\hspace{-3.5mm}
\begin{minipage}[c]{3.8em}\scalebox{1}{
\begin{tikzpicture}[scale=0.3]
\draw [line width=1](0,0)--(0,1)--(1,0)--(1,1);
\draw (0,0) node [scale=0.3, circle, draw,fill=black]{};
\draw (0,1) node [scale=0.3, circle, draw,fill=black]{};
\draw (1,0) node [scale=0.3, circle, draw,fill=black]{};
\draw (1,1) node [scale=0.3, circle, draw,fill=black]{};
\node [left] at (0,-0.1){\small$1$};
\node [left] at (0,1.1){\small$3$};
\node [right] at (1,-0.1){\small$2$};
\node [right] at (1,1.1){\small$4$};
\end{tikzpicture}
}\end{minipage}
\hspace{-0.4cm}, we have $a(0)=a(1)=1$ and, for $n\geq 2$, $a(n) = 4a(n-1) - 3a(n-2) + 1$, hence giving the formula for $a(n)$ and the g.f.\ in the statement of Theorem~\ref{thm-12}.
\end{thm}

\begin{proof} The base cases are clearly true, so let $n\geq 2$. Let $\pi=\pi_1\cdots\pi_n$ be a $p$-avoiding permutation. Clearly, $\pi_1\in\{n-1,n\}$ or $\pi_2\in\{n-1,n\}$ since otherwise the elements $\pi_1,\pi_2,n-1,n$ in $\pi$ form an occurrence of $p$. 
This explains the term $4a(n-1) - 2a(n-2)$ as we can insert either $n-1$ or $n$ in either the first or the second position in a $p$-avoiding $(n-1)$-permutation and subtract the case counted twice when $\{\pi_1,\pi_2\}=\{n-1,n\}$. However, some of the $n$-permutations counted by the term $4a(n-1)$ are not $p$-avoiding, as they contain an occurrence of $p$ with $n-1$ playing the role of {\boldmath $1$}. We say that a permutation is {\em bad} if $n-1\in \{ \pi_1,\pi_2 \}$, $n\not\in \{ \pi_1,\pi_2 \}$, $\pi$ has an instance of $p$ but the permutation $\pi'$ obtained from $\pi$ by removing $n-1$ and reducing $n$ by $1$ is $p$-avoiding. Let $b(n)$ be the number of bad $n$-permutations. Then
\begin{equation}\label{eq-reccur-rel-1}
a(n) = 4a(n-1) - 2a(n-2)-b(n).
\end{equation}  

Note that any bad $n$-permutation is of the form $(n-1)\pi_2\cdots \pi_i n\pi_{i+2}\cdots\pi_n$, where 
\begin{itemize}
\item $2\leq i\leq n-2$ (that is, $\pi_2\neq n$ and $\pi_n\neq n$), and
\item there exist $x\in\{\pi_2,\ldots,\pi_i\}$ and $y\in\{\pi_{i+2},\ldots,\pi_n\}$ such that $x<y$.
\end{itemize}
Indeed, if $\pi_1=n-1$ and it plays the role of {\boldmath $1$} in an occurrence of $p$ then the role of {\boldmath $3$} must be played by $n$, so that there exist $x$ and $y$ with the properties described above. On the other hand, if $\pi_2=n-1$ and $\pi_2xny$ is an occurrence of $p$ then clearly $x\neq n-2$. But since $\pi'$ is $p$-avoiding, we have $\pi_1=n-2$ or $\pi_3=n-2$.
This leads to a contradiction with $(n-2)x(n-1)y$ being an occurrence of $p$ in $\pi'$ showing that no bad $n$-permutation with $\pi_2=n-1$ exists.

We next show, by induction on $n$, that $b(n)=a(n-2)-1$ (for $n\geq 2$), which gives us the desired result by (\ref{eq-reccur-rel-1}). The three base cases hold as there are no bad $n$-permutations for $n=2,3$, and the only bad $4$-permutation is $3142$. Suppose now that $n\geq 5$ and the statement holds for smaller length permutations.

Since each bad $n$-permutation begins with $n-1$, we can remove this element, replace $n$ by $n-1$ in the obtained permutation, and focus on an equivalent problem of counting $p$-avoiding $(n-1)$-permutations having a subsequence $x(n-1)y$, $x<y$; we call such $(n-1)$-permutations r-bad (standing for ``reduced bad''). Clearly, removing the element $n-1$ from an r-bad permutation results in a $p$-avoiding $(n-2)$-permutation, so we will generate all r-bad permutations from the permutations counted by $a(n-2)$ by inserting the element $n-1$.  

We can see that inserting $n-1$ in the first position does not result in an r-bad permutation. If $n-1$ is inserted in the second position of a $p$-avoiding $(n-2)$-permutation $\pi=\pi_1\cdots\pi_{n-2}$, then the obtained permutation is r-bad if and only if $\pi_1\neq n-2$ (as only then we obtain the subsequence $x(n-1)y$; also, note that no occurrence of $p$ can be introduced in this case).  This gives us $a(n-2)-a(n-3)$ r-bad $(n-1)$-permutations since the number of $p$-avoiding $(n-2)$-permutations beginning with $n-2$ is clearly $a(n-3)$. 

Suppose now that $n-1$ is inserted in position $i$, $3\leq i\leq n-2$. Since the result of the insertion is a $p$-avoiding permutation $\pi'=\pi'_1\cdots\pi'_{n-1}$, we see that either $\pi'_1=n-2$ or  $\pi'_2=n-2$ (otherwise, $\pi'_1,\pi'_2,n-2,n-1$ form an occurrence of $p$). We consider these two cases. 

\begin{itemize}
\item $\pi'_1=n-2$. $\pi'$ must contain a subsequence $x(n-1)y$, $x<y$, and clearly $x\neq n-2$. But then $(n-2)x(n-1)y$ is an occurrence of $p$, a contradiction showing that no r-bad permutations can be obtained in this case. 
\item $\pi'_2=n-2$. We claim that removing the element $n-2$ in $\pi'$ and replacing the element $n-1$ by $n-2$, one obtains an r-bad $(n-2)$-permutation $\pi''$. Indeed, $\pi'_2=n-2$ cannot be involved in an occurrence of $p$, or be the $x$ in the subsequence $x(n-1)y$, $x<y$. Hence, $\pi'$ is r-bad if an only if $\pi''$ is r-bad, and therefore the number of bad permutations $\pi'$ of this type is given by $b(n-1)$, which is $a(n-3)-1$ by the induction hypothesis. 
\end{itemize}
To summarise the cases above, $b(n)=a(n-2)-a(n-3)+a(n-3)-1=a(n-2)-1$, as desired.
\end{proof}

\begin{thm}\label{former-conj-1} For the POP $p=$\hspace{-3.5mm}
\begin{minipage}[c]{3.8em}\scalebox{1}{
\begin{tikzpicture}[scale=0.3]
\draw [line width=1](0,0)--(0,1)--(1,0)--(1,1);
\draw (0,0) node [scale=0.3, circle, draw,fill=black]{};
\draw (0,1) node [scale=0.3, circle, draw,fill=black]{};
\draw (1,0) node [scale=0.3, circle, draw,fill=black]{};
\draw (1,1) node [scale=0.3, circle, draw,fill=black]{};
\node [left] at (0,-0.1){\small$1$};
\node [left] at (0,1.1){\small$4$};
\node [right] at (1,-0.1){\small$3$};
\node [right] at (1,1.1){\small$2$};
\end{tikzpicture}
}\end{minipage}
\hspace{-0.4cm}, we have $a(0)=a(1)=1$ and, for $n\geq 2$, $a(n) = 4a(n-1) - 3a(n-2) + 1$, hence giving the formula for $a(n)$ and the g.f.\ in the statement of Theorem~\ref{thm-12}.
\end{thm}

\begin{proof} Let $b(n)$ be the number of $p$-avoiding  permutations $\pi_1\cdots \pi_n$ such that $\pi_n\neq 1$  and each element to the left of the element 1 (if any) is larger than any element to the right of 1. 

Let us compute $a(n)$. Suppose that $\pi=\pi_1\cdots\pi_n$ is $p$-avoiding. If $\pi_n=1$ or $\pi_n=2$ then $\pi_n$ cannot be involved in an occurrence of $p$. Hence there are $2a(n-1)$ such permutations. Further, if $\pi_n\neq 1$ and $\pi_n\neq 2$ and the elements 1 and 2 are not next to each other in $\pi$, then clearly there is an occurrence of $p$ in $\pi$ involving $1$, $2$ and $\pi_n$, which is impossible.  Consider two cases.

\begin{itemize}
\item Suppose that $\pi_{i-1}=1$ and $\pi_i=2$ for some $i\in [2,n-1]$. Then the permutation $\pi'=\pi'_1\cdots\pi'_{n-1}$ obtained from $\pi$ by removing the element 1 and decreasing all other elements by 1 is $p$-equivalent to $\pi$. Indeed, if  the element 1 plays the role of {\boldmath $1$} in an occurrence of $p$ in $\pi$ then the element $2$ cannot play the role of {\boldmath $2$} in that occurrence since there are no candidates for an element to play the role of {\boldmath $3$} in this case. On the other hand, if the element 1 plays the role of {\boldmath $3$} in an occurrence of $p$ in $\pi$ then the element 2 cannot play the role of {\boldmath  $4$} in that occurrence since there is no candidate to play the role of {\boldmath $1$} (which would have to be less than 2). Hence, in both cases, any time the element 1 is involved in an occurrence of $p$, the element 2 may be involved in the same occurrence of $p$ in $\pi'$. We have $a(n-1)-a(n-2)$ such permutations since in $\pi'_{n-1}\neq 1$. 
\item Suppose that $\pi_{i-1}=2$ and $\pi_i=1$ for $2\leq i<n$. If there exist $k<i-1$ and $j>i$ such that $\pi_k<\pi_j$ then the elements $\pi_k$, $2$, $1$ and $\pi_j$ would form an occurrence of $p$, which is impossible. If such $k<i-1$ and $j>i$ do not exist, then the permutation obtained from $\pi$ by removing the element 1 and reducing every other element by 1 is $p$-equivalent to $\pi$. Indeed, if the element 1 plays the role of {\boldmath $1$} in an occurrence of $p$ then 2 also plays the role of {\boldmath $1$} in the same occurrence of $p$ in $\pi'$. Also, 1 cannot play the role of {\boldmath $3$} since there is no candidate for the role of {\boldmath $1$} in that case (which would have to be less than the element playing the role of {\boldmath $4$}). We have $b(n-1)$ such permutations.
\end{itemize}
From our considerations above, we see that 
\begin{equation}\label{last-added-eq-for-an}
a(n)=3a(n-1)-a(n-2)+b(n-1).
\end{equation}
Next, we computer $b(n)$. Suppose that $\pi=\pi_1\cdots\pi_n$ is $p$-avoiding, $\pi_i=1$ for some $i\in [1,n-1]$ and for all $0\leq k<i<j\leq n$ we have $\pi_k>\pi_j$ where $\pi_0:=\infty$. Let $\pi_j=2$ then by definition $j>i$. We consider three cases.
\begin{itemize}
\item $i=n-1$, $j=n$. Clearly, the permutation obtained by removing 1 and 2 from $\pi$ and reducing every other element by 2 is $p$-equivalent to $\pi$ (none of 1 and 2 can play the role of {\boldmath $4$} in an occurrence of $p$). Therefore, we have $a(n-2)$ such permutations. 
\item $i<n-1$, $j\in[i+2,n-1]$. This case is impossible since then the elements $1$, $\pi_{i+1}$, $2$, $\pi_n$ form an occurrence of $p$. 
\item $i<n-1$, $j\in\{i+1,n\}$.  It is easy to see that the permutation obtained by removing the element 2 from $\pi$ and reducing each element in $[3,n]$ by 1 is $p$-equivalent to $\pi$. Hence, we have $2b(n-1)$ such permutations.   
\end{itemize}
Summarising the three cases above, we have
$$b(n)=2b(n-1)+a(n-2).$$
The recursion for $b(n)$ is the same as that for $c(n)$ in the proof of Theorem~\ref{thm-two-N-patterns}, so we can conclude from that proof that $b(n)=a(n)-2a(n-1)+1$, which gives us the desired recursion for $a(n)$ from (\ref{last-added-eq-for-an}). 
\end{proof}

\subsection{N-patterns in Wilf-equivalence class 2 in Table~\ref{N-patterns-tab}}\label{Wilf-class-2-subsec}

\begin{thm}\label{thm-1423} 
For the POP $p=$ \hspace{-3.5mm}
\begin{minipage}[c]{3.8em}\scalebox{1}{
\begin{tikzpicture}[scale=0.3]
\draw [line width=1](0,0)--(0,1)--(1,0)--(1,1);
\draw (0,0) node [scale=0.3, circle, draw,fill=black]{};
\draw (0,1) node [scale=0.3, circle, draw,fill=black]{};
\draw (1,0) node [scale=0.3, circle, draw,fill=black]{};
\draw (1,1) node [scale=0.3, circle, draw,fill=black]{};
\node [left] at (0,-0.1){\small$1$};
\node [left] at (0,1.1){\small$4$};
\node [right] at (1,-0.1){\small$2$};
\node [right] at (1,1.1){\small$3$};
\end{tikzpicture}
}\end{minipage}
we have $a(0)=a(1)=1$ and, for $n\geq 2$, $a(n) = 4a(n-1) - 3a(n-2) + a(n-3)$, hence giving the formula for $a(n)$ and the g.f.\ in the statement of Theorem~\ref{thm-11}.
\end{thm}

\begin{proof} Let $\pi=\pi_1\cdots\pi_n$ be a $p$-avoiding $n$-permutation. Clearly, $\{\pi_1,\pi_2\}\cap\{n-1,n\}\neq\emptyset$ as otherwise there would be an occurrence of $p$ in $\pi$. Moreover, if $\pi_1=n$ or $\pi_2=n$ then $n$ cannot be involved in an occurrence of $p$ in $\pi$, and hence there are $2a(n-1)$ such permutations. Also, if $\pi_2=n-1$ and $\pi_1\neq n$ then removing $n-1$ from $\pi$ and decreasing $n$ by 1 we obtain a $p$-equivalent permutation $\pi'=\pi'_1\cdots\pi'_{n-1}$, since $n-1$ cannot play the role of {\boldmath $2$} in an occurrence of $p$, while if $n-1$ plays the role of {\boldmath $1$} in an occurrence of $p$ then $\pi_1$ also plays the role of {\boldmath $1$} in an occurrence of $p$ in $\pi'$.  Since $\pi'_1\neq n-1$ we have $a(n-1)-a(n-2)$ such permutations. 

We let $b(n)$ be the number of $p$-avoiding $n$-permutations $\pi$ such that $\pi_1=n-1$ and $\pi_2\neq n$. Then
$$a(n)=3a(n-1)-a(n-2)+b(n).$$
Let us compute $b(n)$. Suppose that $\pi$ is a permutation counted by $b(n)$ such that $\pi_1=n-1$, $\pi_i=n$ for $i>2$ and $\pi_j=n-2$. We consider five possible cases.
\begin{itemize}
\item If $j=2$ and $i=3$ then removing from $\pi$ the elements $n-2$, $n-1$, $n$,  we obtain a $p$-equivalent permutation (none of these elements can play the role of {\boldmath $1$} in an occurrence of $p$). Hence, we have $a(n-3)$ permutations in this case.
\item If $j=2$ and $i>3$ then the permutation $\pi'$ obtained from $\pi$ by removing $n-1$ and decreasing $n$ by 1 is $p$-equivalent to $\pi$. Indeed, if the element $n-1$ plays the role of {\boldmath $1$} in an occurrence of $p$ in $\pi$ then the element $n-2$ cannot play the role of {\boldmath $2$}, and therefore $n-2$ may play the role of {\boldmath $1$} in the same occurrence of $p$ in $\pi'$. Since $i>3$, we have $b(n-1)$ such permutations. 
\item If $2<j<i$ then the elements $n-1$, $\pi_2$, $n-2$, $n$ form an occurrence of $p$, so this case is impossible. 
\item If $3<i<j$ then the elements $\pi_2$, $\pi_3$, $n$, $n-2$ form an occurrence of $p$, so this case is impossible. 
\item If $i=3$ and $j>i$ then removing the elements $n-1$ and $n$ we obtain a $p$-equivalent permutation $\pi'=\pi'_1\cdots\pi'_{n-1}$. Indeed, $n-1$ cannot play the role of {\boldmath $1$} in an occurrence of $p$ in $\pi$  because $n$ cannot play the role of {\boldmath $4$}, and since $n-1$ is not involved in an occurrence of $p$ then $n$ also cannot be involved in an occurrence of $p$. Since $\pi'_1\neq n-2$, we have $a(n-2)-a(n-3)$ such permutations.  
\end{itemize}
Summarizing the cases, we have
$$b(n)=a(n-3)+b(n-1)+ a(n-2)-a(n-3)= b(n-1)+ a(n-2)$$
$$= b(n-2)+ a(n-2)+a(n-3)=\cdots
= \sum_{i=1}^{n-2}a(i).$$
Hence, $$a(n)= 3a(n-1)-a(n-2)+b(n)=  3a(n-1)+ \sum_{i=1}^{n-3}a(i)$$ and
$a(n)-a(n-1)= 3a(n-1)- 3a(n-2) +a(n-3)$ giving the desired result.
\end{proof}

\subsection{N-patterns in Wilf-equivalence class 3 in Table~\ref{N-patterns-tab}}\label{Wilf-class-3-subsec}

\begin{thm}\label{thm-new-wilf-equiv} For the N-bipartite POP $p=$\hspace{-3.5mm}
\begin{minipage}[c]{3.8em}\scalebox{1}{
\begin{tikzpicture}[scale=0.3]
\draw [line width=1](0,0)--(0,1)--(1,0)--(1,1);
\draw (0,0) node [scale=0.3, circle, draw,fill=black]{};
\draw (0,1) node [scale=0.3, circle, draw,fill=black]{};
\draw (1,0) node [scale=0.3, circle, draw,fill=black]{};
\draw (1,1) node [scale=0.3, circle, draw,fill=black]{};
\node [left] at (0,-0.1){\small$3$};
\node [left] at (0,1.1){\small$1$};
\node [right] at (1,-0.1){\small$4$};
\node [right] at (1,1.1){\small$2$};
\end{tikzpicture}
}\end{minipage}  \hspace{-0.4cm},  we have $a(n)=n!$ for $n\leq 3$ and, for $n\geq 4$, $a(n)=3a(n-1)+a(n-2)-a(n-3)$.
\end{thm}

\begin{proof} 
Suppose that $\pi=\pi_1\cdots\pi_n$ is a $p$-avoiding $n$-permutation. Clearly, at least one of $\pi_{n-1}$ and $\pi_n$ is $n-1$ or $n$, since otherwise there is an occurrence of $p$ involving $n-1$ and $n$. There are $2a(n-1)$ permutations in question in which $\pi_n\in\{n-1,n\}$. Hence, letting $b(n)$ (resp., $c(n)$) be the number of $p$-avoiding $n$-permutations with $\pi_n<n-1$ and $\pi_{n-1}=n$ (resp., $\pi_{n-1}=n-1$), we have
$$a(n)=2a(n-1)+b(n)+c(n).$$
Let us compute $b(n)$. Clearly, removing $n$ from $\pi$ we obtain a $p$-equivalent permutation in which the maximum element $n-1$ is not the rightmost. Hence, there are $a(n-1)-a(n-2)$ such permutations. 

Next, we compute $c(n)$. We have $\pi_{n-1}=n-1$ and we let $\pi_j=n-2$ and $\pi_i=n$ for $i\neq n$ and $j\leq n$. We consider three cases.

\begin{itemize}
\item $j=n$. In this case, removing $n-1$ and $n-2$ from $\pi$ and reducing $n$ by 2, we get a $p$-equivalent permutation. Indeed, if $n-2$ plays the role of {\boldmath $4$} in an occurrence of $p$ in $\pi$ then $n$ and $n-1$ must play the roles of {\boldmath $1$} and {\boldmath $2$}, but there is no candidate for the role of {\boldmath $3$}. Also, clearly, $n-1$ cannot play the role of {\boldmath $4$} in an occurrence of $p$ in $\pi$. Hence, there are $a(n-2)$ permutations in this case. 
\item $i=n-2$ and $j<n-2$. In this case, removing $n$ from $\pi$ we get a $p$-equivalent permutation $\pi'=\pi'_1\cdots\pi'_{n-1}$. Indeed, $n$ cannot play the roles of {\boldmath $3$} or {\boldmath $4$} in an occurrence of $p$ in $\pi$, and if $n$ plays the role of {\boldmath $2$}, then $n-1$ must play the role of {\boldmath $3$}, but we have no candidate for the role of {\boldmath $1$}. Also, $\pi'_{n-2}=n-1$ and $\pi'_{n-1}\neq n-2$, so we have $b(n-1)=a(n-2)-a(n-3)$ permutations in this case.
\item $i<n-2$ and $j<n-1$. In this case $\pi$ contains an occurrences of $p$. Indeed, if $i<j$ then $n$, $n-2$, $n-1$, $\pi_n$ form an occurrence of $p$, while if $j<i$ then $n-2$, $n$, $\pi_{n-2}$, $\pi_n$ form an occurrence of $p$. 
\end{itemize} 
Summarising the cases above, $c(n)=2a(n-2)-a(n-3)$ and
$$a(n)=2a(n-1)+b(n)+c(n) =2a(n-1)+ a(n-1)-a(n-2) + 2a(n-2)-a(n-3)$$ 
$$= 3a(n-1)+a(n-2)-a(n-3),$$
as desired.
\end{proof}

\section{POPs defined by disjoint chains}\label{sticks-sec}

The minimal disconnected non-trivial bipartite graph $K_2\cup K_1$ corresponds, up to complementing the labels and taking the mirror image with respect to a horizontal line, to two non-Wilf-equivalent POPs  $p_1=$\hspace{-3.5mm}
\begin{minipage}[c]{3.8em}\scalebox{1}{
\begin{tikzpicture}[scale=0.3]
\draw [line width=1](0,0)--(0,1);
\draw (0,0) node [scale=0.3, circle, draw,fill=black]{};
\draw (0,1) node [scale=0.3, circle, draw,fill=black]{};
\draw (1,0) node [scale=0.3, circle, draw,fill=black]{};
\node [left] at (0,-0.1){\small$1$};
\node [left] at (0,1.1){\small$2$};
\node [right] at (1,-0.1){\small$3$};
\end{tikzpicture}
}\end{minipage}
and $p_2=$\hspace{-3.5mm}
\begin{minipage}[c]{3.8em}\scalebox{1}{
\begin{tikzpicture}[scale=0.3]
\draw [line width=1](0,0)--(0,1);
\draw (0,0) node [scale=0.3, circle, draw,fill=black]{};
\draw (0,1) node [scale=0.3, circle, draw,fill=black]{};
\draw (1,0) node [scale=0.3, circle, draw,fill=black]{};
\node [left] at (0,-0.1){\small$1$};
\node [left] at (0,1.1){\small$3$};
\node [right] at (1,-0.1){\small$2$};
\end{tikzpicture}
}\end{minipage} for which the avoidance formulas $a_1(n)=n$ and $a_2(n)=F_n$, the $n$-th Fibonacci number, are well-known and easy to derive.

In this section we prove some interesting results concerning a generalization of the POP $p_1$. We introduce the notion of a {\em disjoint chains POP}, or {\em DC POP}, that is a POP defined by disjoint chains. DC POPs extend the notion of {\em multi-patterns} introduced in \cite{Kit1} to the case of non-consecutive patterns. A DC POP $[\sigma_1,\ldots,\sigma_m]$ of length $k$, where red$(\sigma_i)\in S_{k_i}$ and $\sum_{i=1}^{m}k_i=k$, is defined by $m$ disjoint chains so that the $i$-th path is labeled by $\{1+\sum_{j=1}^{i-1}k_j,\ldots,\sum_{j=1}^{i}k_j\}$ giving the permutation $\sigma_i$ when reading from top to bottom. For example, [12,43,65] and [231,54,678,(11)9(10)] correspond, respectively, to the POPs

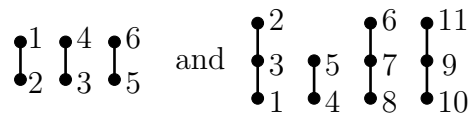
\begin{center}
\begin{tabular}{ccc}
\hspace{-3.5mm}
\begin{minipage}[c]{3.5em}\scalebox{1}{
\begin{tikzpicture}[scale=0.5]

\draw [line width=1](0,-0.5)--(0,0.5);
\draw [line width=1](1.2,-0.5)--(1.2,0.5);
\draw [line width=1](2.5,-0.5)--(2.5,0.5);

\draw (0,-0.5) node [scale=0.4, circle, draw,fill=black]{};
\draw (0,0.5) node [scale=0.4, circle, draw,fill=black]{};
\draw (1.2,0.5) node [scale=0.4, circle, draw,fill=black]{};
\draw (1.2,-0.5) node [scale=0.4, circle, draw,fill=black]{};
\draw (2.5,0.5) node [scale=0.4, circle, draw,fill=black]{};
\draw (2.5,-0.5) node [scale=0.4, circle, draw,fill=black]{};

\node [left] at (0.9,0.6){${\small 1}$};
\node [left] at (0.9,-0.6){${\small 2}$};
\node [right] at (1.2,0.6){${\small 4}$};
\node [right] at (1.2,-0.6){${\small 3}$};
\node [right] at (2.5,0.6){${\small 6}$};
\node [right] at (2.5,-0.6){${\small 5}$};

\end{tikzpicture}
}\end{minipage}
& 
\ \ \ and
&

\hspace{-3.5mm}
\begin{minipage}[c]{3.5em}\scalebox{1}{
\begin{tikzpicture}[scale=0.5]

\draw [line width=1](0,-1.5)--(0,0.5);
\draw [line width=1](1.5,-1.5)--(1.5,-0.5);
\draw [line width=1](3,-1.5)--(3,0.5);
\draw [line width=1](4.5,-1.5)--(4.5,0.5);

\draw (0,-0.5) node [scale=0.4, circle, draw,fill=black]{};
\draw (0,0.5) node [scale=0.4, circle, draw,fill=black]{};
\draw (0,-1.5) node [scale=0.4, circle, draw,fill=black]{};
\draw (1.5,-1.5) node [scale=0.4, circle, draw,fill=black]{};
\draw (1.5,-0.5) node [scale=0.4, circle, draw,fill=black]{};
\draw (3,-0.5) node [scale=0.4, circle, draw,fill=black]{};
\draw (3,0.5) node [scale=0.4, circle, draw,fill=black]{};
\draw (3,-1.5) node [scale=0.4, circle, draw,fill=black]{};
\draw (4.5,-0.5) node [scale=0.4, circle, draw,fill=black]{};
\draw (4.5,0.5) node [scale=0.4, circle, draw,fill=black]{};
\draw (4.5,-1.5) node [scale=0.4, circle, draw,fill=black]{};

\node [left] at (1,-0.6){${\small 3}$};
\node [left] at (1,0.6){${\small 2}$};
\node [left] at (1,-1.6){${\small 1}$};

\node [right] at (1.5,-0.6){${\small 5}$};
\node [right] at (1.5,-1.6){${\small 4}$};

\node [left] at (4,0.6){${\small 6}$};
\node [left] at (4,-0.6){${\small 7}$};
\node [left] at (4,-1.6){${\small 8}$};

\node [left] at (5.9,0.6){${\small 11}$};
\node [left] at (5.6,-0.6){${\small 9}$};
\node [left] at (5.9,-1.6){${\small 10}$};

\end{tikzpicture}
}\end{minipage}
\end{tabular}
\end{center}

\noindent
Clearly, a DC POP with each $\sigma_i$ being of length 2 is a bipartite POP.

In this section we deal with the exponential generating function (e.g.f.) for a sequence $a(n)$, which is $A(x)=\sum_{n\geq 0}a(n)\frac{x^n}{n!}$. A permutation $\pi=\pi_1\cdots\pi_n$ {\em quasi-avoids} a POP $p$ if it contains at least one occurrence of $p$ but the permutation red($\pi_1\cdots\pi_{n-1}$) contains no occurrences of $p$. For example, the permutation $\pi=416532$ quasi-avoids the complete bipartite POP $p=(\{1,2\}/\{3,4\})$ since $\pi$ contains three occurrences of $p$ while the permutation red(41653)=31542 avoids $p$. We let $a^{\star}(n)$ denote the number of $n$-permutations quasi-avoiding the pattern in question and $A^{\star}(x)$ is the respective e.g.f.. 
The notion of quasi-avoidance was first introduced in \cite{Kit1} for so-called {\em consecutive patterns} (that is, patterns whose occurrences form contiguous subsequences in permutations). Essentially copy/pasting the arguments from \cite{Kit1} for consecutive patterns, we can derive the following results. 

\begin{prop}\label{propQuasi} Let $p$ be a POP and $A(x)$ (resp. $A^{\star}(x)$) is the e.g.f. for the number of permutations that avoid (resp. quasi-avoid) $p$. Then, 
$$A^{\star}(x) = (x-1)A(x) + 1.$$ \end{prop}

\begin{proof} We first show that
\begin{equation}
a^{\star}(n) = n\cdot a(n-1) - a(n).
\label{equation} 
\end{equation}

\noindent
Indeed, extending, in all possible ways, each $(n-1)$-permutation avoiding $p$ to the right by adjoining an element gives $n\cdot a(n-1)$ $n$-permutations. The set of these permutations is a disjoint union of the set of all $n$-permutations that avoid $p$ and the set of all $n$-permutations that quasi-avoid $p$, which gives (\ref{equation}). Multiplying both sides of (\ref{equation}) by $x^n/n!$ and summing over all $n$, as well as observing that $a^{\star}(0) = 0$, we obtain the desired result. \end{proof}

\begin{thm}\label{mainAuxilThm} Let $p$ be an arbitrary POP, and the POP $p'$ is obtained from $p$ by increasing each label by $k$ and adding a disjoint path labeled by a permutation $\sigma$ of $\{1,\ldots,k\}$ (from top to bottom). Let $C(x)$ (resp., $A(x)$, $B(x)$) be the e.g.f. for the number of permutations avoiding $p'$ (resp., $\sigma$, $p$). 
Then,
$$C(x)=A(x) + B(x)A^{\star}(x).$$ \end{thm}

\begin{proof} Let $a(n)$, $b(n)$, $c(n)$ be the number of $n$-permutations avoiding the patterns $\sigma$, $p$ and $p'$, respectively. 
If a permutation $\pi=\pi_1\cdots\pi_n$ avoids~$\sigma$ then it avoids~$p'$. Otherwise, $\pi_1\cdots\pi_i$ quasi-avoids $\sigma$ for a unique $i$, $k\leq i\leq n$. But then $\pi_{i+1}\pi_{i+2}\cdots \pi_n$ must avoid $p$. Hence,
\begin{equation}\label{eq-1}
c(n) = a(n) + \displaystyle\sum_{i=0}^{n}{n \choose i}a^{\star}(i)b(n-i)
\end{equation}
since $a^{\star}(i)=0$ for $i=0, \ldots, k-1$. Multiplying both sides of (\ref{eq-1}) by $x^n/n!$ and taking the sum over all~$n$ we get the desired result. \end{proof}

The following corollary of Theorem~\ref{mainAuxilThm} shows that all DC bipartite POPs are Wilf-equivalent. 

\begin{coro}\label{usefulCor}
Let $p=[\sigma_1,\ldots,\sigma_m]$  be a bipartite POP of length $2m$ and $C(x)$ is the e.g.f. for the number of $p$-avoiding permutations. Then,
$$C(x)=\frac{1-(1+(x-1)e^x)^m}{1-x}.$$ 
\end{coro}

\begin{proof} 
We use Theorem~\ref{mainAuxilThm}, induction on~$m$ and the fact that $A(x)=e^x$ (there is only one $\sigma_i$-avoiding $n$-permutation for each $n\geq 0$) and thus by Proposition~\ref{propQuasi}, $A^{\star}(x)=1+(x-1)e^x$. \end{proof}

Corollary~\ref{usefulCor} can be generalized as follows.

\begin{thm}\label{mainThmForMultiPatterns} Let $p=[\sigma_1,\ldots,\sigma_m]$  be a DC POP  and $A_i(x)$ is the e.g.f. for the number of ${\sigma}_i$-avoiding permutations. Then, the e.g.f.\ $A(x)$ for the number of $p$-avoiding permutations is
$$A(x)=\displaystyle\sum_{i=1}^{m}A_i(x)\displaystyle\prod_{j=1}^{i-1}((x-1)A_j(x) + 1).$$\end{thm}

\begin{proof} We use Theorem~\ref{mainAuxilThm} and prove by induction on~$k$ that
$$A(x)=\displaystyle\sum_{i=1}^{m}A_i(x)\displaystyle\prod_{j=1}^{i-1}A^{\star}_j(x).$$
The rest is given by an application of Proposition~\ref{propQuasi}. \end{proof}

\section{Concluding remarks}\label{final-sec}

There is of a host of research directions that can be taken based on our paper, and we will briefly discuss here just some of them. 

Can one generalize/extend Theorems~\ref{t3} and~\ref{t4} dealing with $A=\{i,i+1\}$ and $A=\{i,i+2\}$, respectively, to the case of $A=\{i,i+s\}$ for any $s\geq 1$?

The Wilf-equivalent classes in Table~\ref{N-patterns-tab} raise a number of interesting questions. Indeed, while some of our proofs for the same recurrence relation look quite similar (e.g. those of Theorems~\ref{thm-two-N-patterns} and \ref{former-conj-1}), some others are rather different (e.g. those of Theorems~\ref{thm-two-N-patterns-2} and~\ref{thm-N-pattern-1111}). That would be interesting to find direct (simple) bijections (rather than recursive bijections) explaining (some of) Wilf-equivalences from different rows in Table~\ref{N-patterns-tab}.  

Another natural research direction is in extending our classification and/or enumerative results for N-patterns to POPs defined by 
paths of the forms

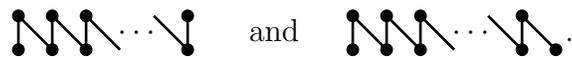
\begin{center}
\begin{tabular}{ccccc}
\hspace{-3.5mm}
\begin{minipage}[c]{3.5em}\scalebox{1}{
\begin{tikzpicture}[scale=0.45]

\draw [line width=1](0,-0.5)--(0,0.5);
\draw [line width=1](1,-0.5)--(1,0.5);
\draw [line width=1](2,-0.5)--(2,0.5);
\draw [line width=1](5,-0.5)--(5,0.5);

\draw [line width=1](1,-0.5)--(0,0.5);
\draw [line width=1](2,-0.5)--(1,0.5);
\draw [line width=1](3,-0.5)--(2,0.5);
\draw [line width=1](5,-0.5)--(4,0.5);

\draw (0,-0.5) node [scale=0.4, circle, draw,fill=black]{};
\draw (0,0.5) node [scale=0.4, circle, draw,fill=black]{};
\draw (1,0.5) node [scale=0.4, circle, draw,fill=black]{};
\draw (1,-0.5) node [scale=0.4, circle, draw,fill=black]{};
\draw (2,0.5) node [scale=0.4, circle, draw,fill=black]{};
\draw (2,-0.5) node [scale=0.4, circle, draw,fill=black]{};
\draw (5,0.5) node [scale=0.4, circle, draw,fill=black]{};
\draw (5,-0.5) node [scale=0.4, circle, draw,fill=black]{};

\node [left] at (4.5,0){$\cdots$};

\end{tikzpicture}
}\end{minipage}
& \hspace{1cm} & and
 &  &
\hspace{-3.5mm}\begin{minipage}[c]{3.5em}\scalebox{1}{
\begin{tikzpicture}[scale=0.45]

\draw [line width=1](0,-0.5)--(0,0.5);
\draw [line width=1](1,-0.5)--(1,0.5);
\draw [line width=1](2,-0.5)--(2,0.5);
\draw [line width=1](5,-0.5)--(5,0.5);

\draw [line width=1](1,-0.5)--(0,0.5);
\draw [line width=1](2,-0.5)--(1,0.5);
\draw [line width=1](3,-0.5)--(2,0.5);
\draw [line width=1](5,-0.5)--(4,0.5);
\draw [line width=1](6,-0.5)--(5,0.5);

\draw (0,-0.5) node [scale=0.4, circle, draw,fill=black]{};
\draw (0,0.5) node [scale=0.4, circle, draw,fill=black]{};
\draw (1,0.5) node [scale=0.4, circle, draw,fill=black]{};
\draw (1,-0.5) node [scale=0.4, circle, draw,fill=black]{};
\draw (2,0.5) node [scale=0.4, circle, draw,fill=black]{};
\draw (2,-0.5) node [scale=0.4, circle, draw,fill=black]{};
\draw (5,0.5) node [scale=0.4, circle, draw,fill=black]{};
\draw (5,-0.5) node [scale=0.4, circle, draw,fill=black]{};
\draw (6,-0.5) node [scale=0.4, circle, draw,fill=black]{};

\node [left] at (4.5,0){$\cdots$};

\end{tikzpicture}
}\end{minipage}\hspace{1.6cm}. 
\end{tabular}
\end{center}
In particular, the next natural step is in studing such POPs of length 5 that we call {\em $M$-patterns}. For example, the avoidance sequence for the M-pattern \hspace{-3.5mm}
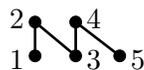
\begin{minipage}[c]{3.5em}\scalebox{1}{
\begin{tikzpicture}[scale=0.45]

\draw [line width=1](0,-0.5)--(0,0.5);
\draw [line width=1](1.2,-0.5)--(1.2,0.5);
\draw [line width=1](1.2,-0.5)--(0,0.5);
\draw [line width=1](2.5,-0.5)--(1.2,0.5);

\draw (0,-0.5) node [scale=0.4, circle, draw,fill=black]{};
\draw (0,0.5) node [scale=0.4, circle, draw,fill=black]{};
\draw (1.2,0.5) node [scale=0.4, circle, draw,fill=black]{};
\draw (1.2,-0.5) node [scale=0.4, circle, draw,fill=black]{};
\draw (2.5,-0.5) node [scale=0.4, circle, draw,fill=black]{};

\node [left] at (0,0.6){{\small$2$}};
\node [left] at (0,-0.6){{\small$1$}};
\node [right] at (1.2,0.6){{\small$4$}};
\node [right] at (1.2,-0.6){{\small$3$}};
\node [right] at (2.5,-0.6){{\small$5$}};

\end{tikzpicture}
}\end{minipage}\hspace{0.7cm} begins with 1, 2, 6, 24, 104, 448, 1888, 7808, 31872, ... and this sequence is not in \cite{oeis}. For another example, the avoidance sequence for the M-pattern \hspace{-3.5mm}
\begin{minipage}[c]{3.5em}\scalebox{1}{
\begin{tikzpicture}[scale=0.45]

\draw [line width=1](0,-0.5)--(0,0.5);
\draw [line width=1](1.2,-0.5)--(1.2,0.5);
\draw [line width=1](1.2,-0.5)--(0,0.5);
\draw [line width=1](2.5,-0.5)--(1.2,0.5);

\draw (0,-0.5) node [scale=0.4, circle, draw,fill=black]{};
\draw (0,0.5) node [scale=0.4, circle, draw,fill=black]{};
\draw (1.2,0.5) node [scale=0.4, circle, draw,fill=black]{};
\draw (1.2,-0.5) node [scale=0.4, circle, draw,fill=black]{};
\draw (2.5,-0.5) node [scale=0.4, circle, draw,fill=black]{};

\node [left] at (0,0.6){{\small$1$}};
\node [left] at (0,-0.6){{\small$3$}};
\node [right] at (1.2,0.6){{\small$2$}};
\node [right] at (1.2,-0.6){{\small$4$}};
\node [right] at (2.5,-0.6){{\small$5$}};

\end{tikzpicture}
}\end{minipage}\hspace{0.7cm} begins with 1, 2, 6, 24, 104, 448, 1888, 7808, 31872, ... and this sequence is also not in \cite{oeis}.

Finally, that could be interesting to obtain (general) enumerative/Wilf-equivalence results for POPs defined by bipartite graphs of some regular shapes, e.g. such as 

\begin{center}
\hspace{-3.5mm}
\begin{minipage}[c]{3.5em}\scalebox{1}{
\begin{tikzpicture}[scale=0.45]

\draw [line width=1](0,-0.5)--(0,0.5);
\draw [line width=1](1,-0.5)--(1,0.5);
\draw [line width=1](2,-0.5)--(2,0.5);
\draw [line width=1](5,-0.5)--(5,0.5);

\draw [line width=1](1,-0.5)--(0,0.5);
\draw [line width=1](0,-0.5)--(1,0.5);
\draw [line width=1](1,-0.5)--(2,0.5);
\draw [line width=1](2,-0.5)--(1,0.5);
\draw [line width=1](2,-0.5)--(3,0.5);
\draw [line width=1](3,-0.5)--(2,0.5);
\draw [line width=1](5,-0.5)--(4,0.5);
\draw [line width=1](4,-0.5)--(5,0.5);

\draw (0,-0.5) node [scale=0.4, circle, draw,fill=black]{};
\draw (0,0.5) node [scale=0.4, circle, draw,fill=black]{};
\draw (1,0.5) node [scale=0.4, circle, draw,fill=black]{};
\draw (1,-0.5) node [scale=0.4, circle, draw,fill=black]{};
\draw (2,0.5) node [scale=0.4, circle, draw,fill=black]{};
\draw (2,-0.5) node [scale=0.4, circle, draw,fill=black]{};
\draw (5,0.5) node [scale=0.4, circle, draw,fill=black]{};
\draw (5,-0.5) node [scale=0.4, circle, draw,fill=black]{};

\node [left] at (4.5,0){$\cdots$};

\end{tikzpicture}
}\end{minipage}\hspace{1.2cm}. 
\end{center}

\end{document}